\documentclass[a4paper, 11pt]{article}

\usepackage{amsthm}
\usepackage{amsmath}
\usepackage{mathtools} %lädt auch amsmath
\usepackage{amssymb}
\usepackage{graphicx}
\usepackage[round]{natbib}
\usepackage{bbm}
\usepackage{textcomp}
\usepackage{url}
\usepackage[margin=2.5cm]{geometry}
\usepackage{mathrsfs} % WT Notation
\usepackage{xcolor}
\usepackage{float}
\usepackage{picinpar}
\usepackage{subfigure}

\theoremstyle{plain}  %default 
\newtheorem{thm}{Theorem}[section] 
\newtheorem{lem}[thm]{Lemma} 
\newtheorem{prop}[thm]{Proposition}

\theoremstyle{plain} % just in case the style had changed
\newcommand{\thistheoremname}{}
\newtheorem{genericthm}[thm]{\thistheoremname}

\newtheorem*{genericthm*}{\thistheoremname}
\newenvironment{namedthm*}[1]
{\renewcommand{\thistheoremname}{#1}%
	\begin{genericthm*}}
	{\end{genericthm*}}

\theoremstyle{definition} 
\newtheorem{defn}{Definition}[section] 
 
\newtheorem{exmp}{Example}[section]
\newtheorem{assump}{Assumption}[section]

\theoremstyle{remark} 
\newtheorem{rem}{Remark}

\newcommand{\prob}{\mathbb{P}} % probability measure
\newcommand{\E}{\mathbb{E}} % expectation
\renewcommand{\O}{\Omega} % Probability state space 
\newcommand{\EE}{\mathcal{E}} % generic generator
\newcommand{\F}{\mathcal{F}} % generic sigma algebra
\newcommand{\A}{\mathcal{A}} % generic sigma lattice
\newcommand{\M}{\mathcal{M}} % generic sigma lattice 2
\newcommand{\LL}{\mathscr{L}} % lattice 

\newcommand{\R}{\mathbb{R}} % reals
\newcommand{\N}{\mathbb{N}} % naturals
\newcommand{\Q}{\mathbb{Q}} % rationals

\newcommand{\diff}{\,\mathrm{d}} % differential 
\newcommand{\one}{\mathbbm{1}} % indicator

\newcommand{\s}{\sigma} % sigma

\newcommand{\RR}{\mathfrak{R}} % ring of half open intervals

\newcommand{\BR}{\mathcal{B}(\mathbb{R})} % Borel-sigma-algebra on R
\newcommand{\PR}{\mathcal{P}(\mathbb{R})} % space of prob measures
\newcommand{\HH}{\mathscr{H}} % space of Markov kernels
\newcommand{\GG}{\mathscr{G}} % space of measurable Markov kernels

\newcommand{\st}{\leq_{st}} % stochastic order

\newcommand{\X}{\boldsymbol{X}} % random vector X
\renewcommand{\a}{\boldsymbol{a}} % vector a
\renewcommand{\b}{\boldsymbol{b}} % vector b

\newcommand{\XX}{\mathcal{X}} % curly X
\newcommand{\BX}{\mathcal{B}(\mathcal{X})} % Borel sigma algebra on X
\newcommand{\Tau}{\mathcal{T}} % topologies
\newcommand{\U}{\mathscr{U}} % uppersets 
\newcommand{\OO}{\mathscr{O}} % open sets 

\newcommand{\bx}{\pmb{x}} % covariate
\newcommand{\by}{\pmb{y}} % response vector
\newcommand{\myS}{{\rm S}} % generic scoring rule
\newcommand{\crps}{\textrm{CRPS}} % crps
\newcommand{\QS}{\textrm{QS}} % quantile score
\newcommand{\BS}{\textrm{BS}} % brier score

\newcommand{\eqforall}{\quad \textrm{for all }}
\newcommand{\eqand}{\quad \textrm{and}\quad}

\newcommand\myrot[1]{\mathrel{\rotatebox[origin=c]{#1}{$\Rightarrow$}}}
\newcommand\SWarrow{\myrot{-135}}
\newcommand\SEarrow{\myrot{-45}}

\newcommand{\hsp}{\hspace{0.2mm}}
\newcommand{\Mid}{\hspace{1mm} \big| \hspace{1mm}}
\newcommand{\MMid}{\hspace{1mm} \Big| \hspace{1mm}}

\begin{document}
	\title{Isotonic conditional laws}
	\author{Sebastian Arnold and Johanna Ziegel\footnote{IMSV, University of Bern, \texttt{\{sebastian.arnold,johanna.ziegel\}@unibe.ch}}}
	\maketitle
 
 \begin{abstract}
We introduce isotonic conditional laws (ICL) which extend the classical notion of conditional laws by the additional requirement that there exists an isotonic relationship between the random variable of interest and the conditioning random object. We show existence and uniqueness of ICL building on conditional expectations given $\s$-lattices. ICL corresponds to a classical conditional law if and only if the latter is already isotonic. ICL is motivated from a statistical point of view by showing that ICL emerges equivalently as the minimizer of an expected score where the scoring rule may be taken from a large class comprising the continuous ranked probability score (CRPS). Furthermore, ICL is calibrated in the sense that it is invariant to certain conditioning operations, and the corresponding event probabilities and quantiles are simultaneously optimal with respect to all relevant scoring functions. We develop a new notion of general conditional functionals given $\s$-lattices which is of independent interest. 
 \end{abstract}


\section{Introduction}\label{Introduction}
For a generic random vector $(Y,\xi)$ defined on some probability space $(\Omega, \F, \prob)$, the conditional law of $Y$ given $\xi$ is the distribution of $Y$ given ''all information about $\xi$". Typically, the information carried by $\xi$ is encoded by the $\s$-algebra generated by $\xi$. %, for short $\s(\xi)$. 
In some situations such as isotonic regression, it is reasonable to assume that there exists an \emph{isotonic} relationship between the random element $\xi$ and the random variable $Y$ in the sense that large (small) realizations of $\xi$ make it more likely that $Y$ will be large (small) as well. Evidently, in order to be able to speak about smaller and larger realizations, we have to assume that the random element $\xi$ takes values in some measurable space $\XX$ which is equipped with a partial order $\preceq$. Throughout the article, we assume that the quantity of main interest, $Y$, is a real-valued random variable. We show that there exists a unique \emph{isotonic conditional law (ICL) of $Y$ given $\xi$}, which takes the form of a Markov kernel and describes the best available explanation of the distribution of $Y$ given all information of $\xi$ under the assumption of isotonicity between $Y$ and $\xi$. 

While classical conditional distributions are constructed using conditional expectations given $\s$-algebras, our theory builds on conditional expectations given $\s$-lattices, which were first examined by \citet[1963, 1965]{Brunk_1961}. A family of sets $\A \subseteq \F$ is called a \emph{$\s$-lattice} if it contains $\Omega$, the empty set, and is closed under taking countable unions and intersections. In particular, any $\s$-algebra is a $\s$-lattice. In the Hilbert space of square integrable random variables, the family of all random variables which are measurable with respect to some $\s$-lattice $\A$ form a closed convex cone, and the conditional expectation of any random variable $X$ with respect to $\A$, for short $\E(X \mid \A)$, is defined as the orthogonal projection onto this closed convex cone \citep{Brunk_1965}. Conditional expectations given $\s$-lattices are strongly connected to isotonicity which will be illustrated in detail in Section \ref{Sec:Factorization_and_connection_to_classical_conditional_laws}. In particular, the conditional expectation factorizes with respect to an increasing measurable function, that is, there exists an increasing measurable function $f:\XX \to \R$ such that
\begin{equation*}
    \E(Y \mid \LL(\xi))=f(\xi) \quad \textrm{almost surely},
\end{equation*}
where $\LL(\xi)\subseteq \F$ denotes the \emph{$\s$-lattice generated by $\xi$}, see Section \ref{sec:generated_lattices}. This result is analogous to the well-known factorization result for classical conditional expectations given $\s$-algebras \citep[Lemma 11.7]{Bauer_MI}, with the difference that, additionally, $f$ has to be increasing. %In this paper, we firstly study isotonic conditional laws from a probabilistic perspective, and secondly from a statistical point of view. 
While the factorization result for conditional expectations given $\s$-lattices is already known in case of $\xi$ being a random variable, see e.g. \cite{Chu-in}, we extend the result to more general random elements in Section \ref{Subsec:Factorization}. Moreover, we will show that the isotonic conditional law of $Y$ given $\xi$ coincides with the conditional law of $Y$ given $\xi$ if and only if the latter is already isotonic in $\xi$.

For a random variable $Y$ and a $\s$-lattice $\A \subseteq \F$, we show in Theorem \ref{thm:universality_of_ICL} that the ICL $P_{Y \mid \A}$ minimizes the expected score $\E \myS (F,Y$) over all random distributions $F$ which are isotonic, that is, $\A$-measurable, and where the scoring rule $\myS$ may be chosen from a class of proper scoring rules including the prominent continuous ranked probability score (CRPS; \citealp{Matheson1976,Gneiting_Raftery_2007}). This universal minimizing property of ICL shows that ICL can be interpreted as the theoretically optimal prediction for the distribution of $Y$ given all information contained in $\A$, or, as the theoretically optimal prediction for the distribution of $Y$ given all information of $\xi$ under the assumption of isotonicity between $Y$ and $\xi$ if the $\s$-lattice $\A$ is generated by the random element $\xi$. Moreover, in Section \ref{Subsec:ICL_as_a_CRPS_minimizer}, we show that ICL is the population version of isotonic distributional regression (IDR) as proposed by \cite{IDR}, see also \citet{El-BarmiMukerjee2005,Mosching2022}. IDR is a powerful nonparametric distributional regression method under an isotonicity constraint that has found several applications in the statistical literature, see e.g.\ \cite{Schulz_Lerch_2022}, \cite{Alexander_Johanna_Biometrika}, \cite{DIM}. 

A probabilistic forecast for $Y$ is a (random) predictive distribution $F$ that aims to be as close as possible to the conditional law of $Y$ given the information at the time of forecasting. Calibration is a key requirement for probabilistic forecasts. Roughly speaking, probabilistic forecasts are \emph{calibrated} (or \emph{reliable}) if they are statistically compatible with the observations \citep{Gneiting_Katzfuss_2014}. Numerous different notions of calibration are found in the literature, see e.g.\ \cite{Dawid1984} and \cite{Diebold1998} for probabilistic calibration, \cite{Strahl_Ziegel_2017} and \cite{Arnold_Henzi_Ziegel} for conditional and sequential notions of calibration, \cite{Thorarinsdottir_et_al_2016} and \cite{Ziegel_Gneiting_2014} for multivariate notions of calibration, and \cite{Guo_et_al_2017} and \cite{Gupta_et_al_2022} for applications in machine learning. A recent extensive study of the logical dependencies between the main notions of calibration is given in \cite{Tilmann_Johannes_Calibration}. 

%In this section, we introduce the new notion of isotonic calibration and study the logical dependencies with respect to the three familiar types of auto-, threshold- and quantile-calibration. Moreover, we show that ICL is threshold- and quantile calibrated, which will particularly be used to show Theorem \ref{thm:universality_of_ICL}.

Analogously to IDR, one can show that ICL automatically has certain calibration properties. We introduce the new notion of \emph{isotonic calibration} which is satisfied by ICL, and show that isotonic calibration is weaker than auto-calibration and stronger than threshold- or quantile-calibration, see Section \ref{Sec:Calibration}. The calibration properties of ICL can alternatively be viewed as invariance properties of ICL under certain conditioning operations. Such invariance properties are well-known for classical conditional laws given $\sigma$-algebras.

IDR has the interesting property that it simultaneously fits optimal increasing quantile curves, and optimal increasing event probabilities. We show that the same property carries over to ICL, that is, for each threshold $z\in \R$, or quantile level $\alpha \in (0,1)$, the $\A$-measurable random variables $P_{Y \mid \A}(z)$ and $P_{Y \mid \A}^{-1}(\alpha)$ are simultaneously optimal with respect to all relevant scoring functions for the mean, or quantile functional, respectively, see Theorem \ref{thm:universality_of_ICL}. The difficulty is here that it is not obvious how one should naturally define conditional functionals, such as quantiles, given $\s$-lattices. 

\cite{Fissler_2022} recently studied measurability conditions for compositions of Markov kernels and statistical functionals. An alternative approach lies in studying conditional functionals directly. Whereas conditional expectations are broadly studied and developed, the theoretical literature about conditional quantiles is by not as extensive. For a random variable $Y$ and a $\s$-algebra $\mathcal{B}\subseteq \F$, \cite{Tomkins} defines the conditional median of $Y$ given $\mathcal{B}$ as any $\mathcal{B}$-measurable random variable $X$ which satisfies $\prob (Y>X\mid \mathcal{B})\leq 1/2 \leq \prob(Y \geq X \mid \mathcal{B})$. Clearly, this definition may be extended to any quantile 
%\cite{Armerin} shows that any conditional quantile can equivalently be characterized as a minimizer of the expected score $\E \QS_\alpha(G,Y)$ over all $\A$-measurable random variables, where $\QS_\alpha$ denotes the quantile score at level $\alpha$. 
but this approach does not readily generalize to $\s$-lattices. Instead, in Section \ref{Sec:Conditional_Functionals}, we define conditional functionals based on identification functions as minimizers of expected scoring functions building on the work of \cite{Optimal}. This approach is closely connected to the approach given by \cite{Brunk1970}, where the authors characterize conditional expectations as generalized Radon-Nikodym derivatives but do not apply their considerations to quantiles. We apply our results on conditional functionals given $\s$-lattices to show the universality properties of ICL but the proposed construction is of independent interest.
%Our study of conditional quantiles allows us to show the universality of ICL, namely that $P_{Y \mid \A}^{-1}(\alpha)$ minimizes $\E \QS_\alpha(G,Y)$ over all $\A$-measurable random variables. Even though we mainly use 

Except for Theorem \ref{thm:universality_of_ICL} and Proposition \ref{prop:ICL_is_quantile_calibrated}, all proofs are collected in the appendix.

\section{Existence and uniqueness of ICL}\label{Sec:Existence_and_Uniqueness}
Throughout the article, we fix a probability space $(\O, \F, \prob)$. Let $\A\subseteq \F$ be a \emph{$\s$-lattice}, that is $\A$ is closed under countable intersections and unions and $\emptyset, \Omega \in \A$. A random variable $X$ is called \emph{$\A$-measurable} if $\{X>x\} \in \A$ for all $x \in \R$. Denote by $\overline{\A}$ the family of all complements of elements in $\A$. It is straightforward to verify that $\overline{\A}$ is also a $\s$-lattice. Let $L_1$ be the space of all integrable random variables, $L_1(\A)$ the family of all $\A$-measurable integrable random variables, and $\BR$ the Borel $\s$-algebra on $\R$. For a random variable $X$, we denote by $\s(X)$ the $\s$-algebra generated by $X$, i.e.~$\s(X) = \{X^{-1}(B) \mid B\in \BR\}$. Following \cite{Brunk_1963}, we define:
    \begin{defn}\label{defn:cond_expectations}
        Let $\A \subseteq \F$ be a $\s$-lattice. For any $Y\in L_1$, we call $X\in L_1(\A)$ a \emph{conditional expectation of $Y$ given $\A$}, for short $\E(Y \mid \A)$, if it satisfies 
		\begin{equation} \label{eq:defn_cond_expectations1}
			\E(Y\one_A)\leq \E(X\one_A)\eqforall A \in \A 
		\end{equation}
		and
    \begin{equation}\label{eq:defn_cond_expectations2}
			\E(X \one_B) = \E(Y \one_B) \eqforall B \in \s(X).
		\end{equation}
    \end{defn}

\cite{Brunk_1963} shows existence and uniqueness of conditional expectations given $\s$-lattices (see Theorems 1 and 2 therein). Hence, we will denote by $\E(Y\mid \A)$ any $\A$-measurable random variable $X$ that satisfies \eqref{eq:defn_cond_expectations1} and \eqref{eq:defn_cond_expectations2} and know that $X=\Tilde{X}$ almost surely for any other random variable $\tilde{X}$ with these properties. We call a conditional expectation given a $\s$-lattice sometimes simply \emph{isotonic} conditional expectation and abbreviate $\E(\one_{B} \mid \A)$ by $\prob(B \mid \A)$ for $B \in \F$.

\begin{exmp}
%[Isotonic conditional expectations as a generalization of isotonic regression] 
%This example shows that isotonic conditional expectations generalize isotonic regression for the mean functional \citep[Chapter 8]{RWD}. 
This example shows that the isotonic conditional expectation of a random variable on a finite and totally ordered probability space coincides with the solution to the isotonic regression problem for the mean functional (\citealp{Ayer1955}; \citealp{4B}).
%, \citealp{RWD}). 

Let $\Omega=\{1,2,3\}$ be endowed with the power set $\F=2^\Omega$ and $\prob(A)=\#A/3$ for $A\in \F$. Consider the $\s$-lattice $\A=\{\emptyset, \Omega, \{2,3\}, \{3\}\}$ and the random variable $Y(i)=y_i$, for $i=1,2,3$ with $y_2<y_1<y_3$. Then, the conditional expectation $X=\E(Y\mid \A)$ is given by $X(1)=X(2)=(y_1+y_2)/2$ and $X(3)=y_3$.
%Since $\E(Y\one_A)< \E(X\one_A)$ for $A=\{2,3\}$, this example shows that the inequality in \eqref{eq:defn_cond_expectations1} may be strict in general.
\end{exmp}

Some properties of classical conditional expectations do not hold any longer if we consider conditional expectations given $\s$-lattices rather than  $\s$-algebras. In particular, we do not have linearity, the tower property fails to hold and $\A$-measurable random variables can not be pulled out of the conditional expectation. Explicit counterexamples can be found e.g.\ in \cite{Chu-in} or \citet{Masterthesis}. For a comprehensive study of sufficient and necessary conditions for the conditional expectation to be linear see \cite{Kuenzi}. \cite{Brunk_1965} shows that $\E(X\mid \A)$ is square integrable for any square integrable random variable $X$ and defines $\E(X\mid \A)$ in this case as the projection of $X$ onto the closed convex cone of $\A$-measurable random variables in the space $L_2$ of square integrable random variables with the inner product $\langle W, Z \rangle = \E(WZ)$ for $W,Z \in L_2$. For square integrable random variables, there exists a nice geometrical interpretation why linearity and the tower property are violated as illustrated in Figure \ref{fig1}. Moreover, the interpretation of isotonic conditional expectations as projections also illustrates why we do not require equality in \eqref{eq:defn_cond_expectations1} but rather only impose that $\E((Y-X)Z)\leq 0$ for all $Z\in L_2(\A)$ (which is equivalent to condition \eqref{eq:defn_cond_expectations1}): Whereas $(Y-X)$ and $Z$ are orthogonal for any $\A$-measurable random variable $Z$ for classical conditional expectations (i.e. $\A$ is a $\s$-algebra), the angle between $(Y-X)$ and $Z$ only has to be obtuse for any $Z\in L_2(\A)$ in the case of $\A$ being a $\s$-lattice.
	
\begin{figure}
	\includegraphics[width=0.49\textwidth]{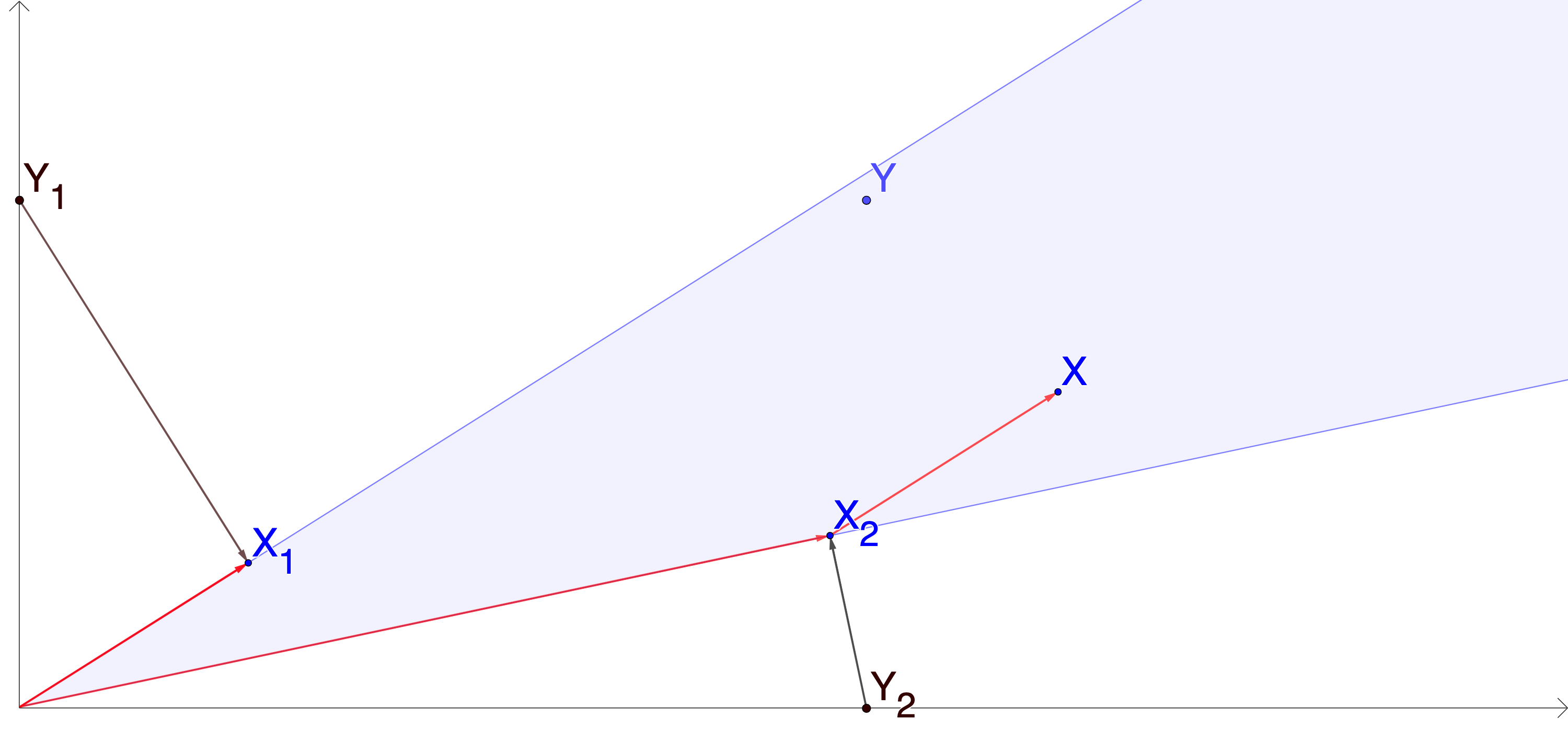}
	\includegraphics[width=0.49\textwidth]{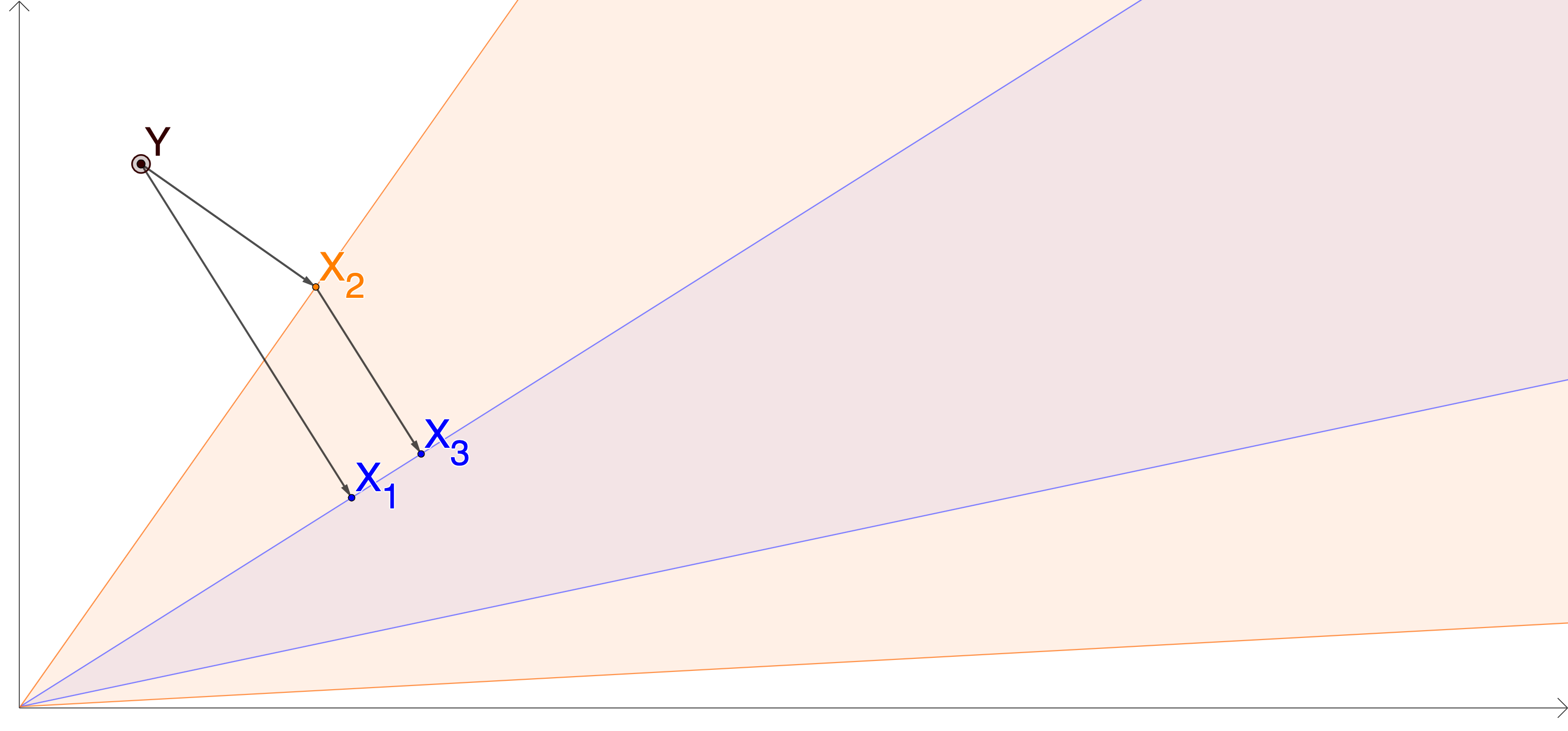}
	\caption{Linearity and the tower property are violated for conditional expectations given $\s$-lattices. Left panel: The blue cone indicates the family $L_2(\A)$ for $\s$-lattice $\A\subseteq \F$. Here, $\E(Y_1 \mid \A)+ \E(Y_2 \mid \A) =X \neq Y=\E(Y_1+Y_2 \mid \A)$. Right panel: For $\s$-lattices $\A_1\subseteq A_2 \subseteq \F$, the families $L_2(\A_1)$ and $L_2(\A_2)$ are denoted by the blue and orange cone, respectively. Clearly, $\E(\E(Y \mid \A_2) \mid \A_1) =X_3\neq X_1=\E(Y \mid \A_1)$.}\label{fig1}
\end{figure}

In this article, we study isotonic conditional laws, that is conditional laws given $\s$-lattices. Our first main result is the following theorem, which guarantees existence and uniqueness. Recall that a \emph{Markov kernel} (also called a  \emph{stochastic kernel} or a \emph{probability kernel}) from $(\Omega, \F)$ to $(\R, \BR)$ is a map $k: \Omega \times \BR \to [0,1]$ such that
 \begin{enumerate}
     \item [(i)] for each $\omega \in \Omega$, the map $B\mapsto k(\omega, B)$, $B\in \BR$, is a probability measure on $(\R, \BR)$,
     \item [(ii)] for each $B \in \BR$, the map $\omega \mapsto k(\omega,B)$, $\omega \in \Omega$, is $\F$-$\BR$-measurable.  
 \end{enumerate}
	
\begin{thm}[Existence and uniqueness of ICL] \label{thm:existence_and_uniqueness_ICL}
Let $Y$ be a random variable and $\A \subseteq \F$ be a $\s$-lattice. Then there exists a conditional law of $Y$ given $\A$, denoted by $P_{Y\mid \A}$.  More precisely, there exists a Markov kernel $P_{Y \mid \A}$ from $(\Omega, \F)$ into $(\R, \BR)$ such that $\omega \mapsto P_{Y \mid \A}(\omega,(y, \infty))$ is a version of $\prob(Y>y \mid \A)$ for any $y \in \R$. For any other Markov kernel $Q$ with this property, it holds that $P_{Y\mid \A}(\cdot,B)=Q(\cdot,B)$ $\prob$-almost surely for all $B \in \BR$.
\end{thm}

We refer to the conditional law of $Y$ given $\A$ as the \emph{isotonic conditional law (ICL) of $Y$ given $\A$}. By establishing the notion of a $\s$-lattice that is generated by some random element $\xi$ it becomes clear why we refer to ICL as the \emph{isotonic} conditional law of $Y$ given $\xi$. Theorem \ref{thm:existence_and_uniqueness_ICL} is analogous to the well-known result about existence and uniqueness of conditional laws given $\s$-algebras and is shown in a similar way, see e.g.~\citet[Theorem 4.33]{Bauer_Prob}. For this reason we only give a sketch of proof in Appendix \ref{appendix:proof_of_existence_and_uniqueness_of_ICL} and refer to the Master's thesis by \citet{Masterthesis} for a detailed proof. 

\begin{rem}
    We only require $P_{Y \mid \A}(\cdot, (a,\infty))=\prob(Y>a\mid \A)$ for all $a\in \R$ and not $P_{Y \mid \A}(\cdot, B)=\prob(Y\in B\mid \A)$ for any $B \in \BR$ for the following reason: For any $B\in \BR$, we have $ \prob(Y\in B^c \mid \A)= 1-\prob(Y\in B \mid \overline{\A})$ by the general properties of isotonic conditional expectations, see e.g.\ \cite{Brunk_1963}. If we imposed the stronger condition on $P_{Y \mid \A}$, then also $ \prob(Y\in B^c \mid \A)= 1-\prob(Y\in B \mid \A)$ since $P_{Y\mid \A}(\cdot, B)=1-P_{Y\mid \A}(\cdot, B^c)$ almost surely for any $B\in \BR$. This would imply $\prob(Y\in B \mid \A)=\prob(Y\in B \mid \overline{\A})$ for all $B\in \BR$ which holds if and only if $\A$ is a $\s$-algebra as explained by \cite{Masterthesis}.
 \end{rem}

%To develop a meaningful notion of generated $\s$-lattices for random vectors and random cdf's, a comprehensive study of ordered metric spaces is necessary. The next section summarizes as briefly as possible how we should define generated $\s$-lattices for random vectors and random cdf's. The rigorous theoretical background is given in the appendix \ref{Appendix:Ordered_Metric_Spaces}. 

\section{Factorization and connection to classical conditional laws}\label{Sec:Factorization_and_connection_to_classical_conditional_laws}

\subsection{Generated $\s$-lattices and measurability for random elements}\label{sec:generated_lattices}
We develop a notion of generated $\s$-lattices for random elements. 
%In order to speak about measurability and isotonicity, we will need an order and a topology or metric on $\XX$. 
Our primary interest lies in random vectors $\X\in \R^p$ and random cdfs $F\in \PR$, but we will study more general random elements in the next section and illustrate in Examples \ref{ex:Euclidean_space} and \ref{ex:The_space_PR} the connection to these two special cases. 

Here, $\PR$ denotes the set of all probability measures on $\R$. We treat elements of $\PR$ interchangeably as probability measures or cumulative distribution functions (cdfs) and let $P$ and $Q$ be generic distributions and $F$ and $G$ be generic cdfs. For a cdf $F$, we denote the corresponding survival function (sf) by $\Bar{F}$, i.e.\ $\Bar{F}(x)=1-F(x)$ for $x\in \R$.

%From now on we are interested in random elements $\xi$ that map from the underlying probability space $(\Omega, \F, \prob)$ into some nonempty space $\XX$. To be able to speak about measurability and isotonicity, we will need a partial order and a metric on $\XX$.  either into $\R^p$ or into the space of all probability distributions on $\R$, which we denote by $\PR$. That is, we consider either random vectors $\X$ in case of $\XX= \R^p$ or random cdfs $F$ in case of $\XX= \PR$. \\
For a nonempty set $\XX$ equipped with a partial order $\preceq$ and a metric $d$, a triple $(\XX, d, \preceq)$ is called an \emph{ordered metric space}, see \citet[Chapter 10]{Richmond} for a comprehensive study of ordered metric spaces. %The following examples are of main interest. 
Let $(\XX, d, \preceq)$ be an ordered metric space, and denote the power set of $\XX$ by $2^\XX$. 
%A family of sets $\A \subseteq 2^\XX$ is called \emph{$\s$-lattice} if $\emptyset, \XX \in \A$ and if it is closed under countable intersections and countable unions. 
An arbitrary intersection of $\s$-lattices is again a $\s$-lattice. Hence, for any family of sets $\EE\subseteq 2^\XX$, we can define the \emph{$\s$-lattice generated by $\mathcal{E}$} as the smallest $\s$-lattice that contains $\mathcal{E}$ and denote it by $\LL(\mathcal{E})$. Let $\BX$ denote the Borel $\s$-algebra on $\XX$, that is the $\s$-algebra generated by all open sets and call $B\subseteq \XX$ an \emph{upper set} if $x \in B$ and $x \preceq y$ implies $y \in B$. We denote by $\U$ the family of all upper sets and leave it to the reader to show that $\U$ is a $\s$-lattice. 

\begin{exmp}[Euclidean space]\label{ex:Euclidean_space}
    Consider $\XX=\R^p$ with the usual \emph{componentwise order} $\a \leq \b$ if $a_i \leq b_i$ for all $1\leq i \leq p$ and the \emph{Euclidean metric} $d(\a, \b)= \lVert \a - \b \rVert$ for $\a=(a_i)_{i=1}^p, \b=(b_i)_{i=1}^p \in \R^p$.
\end{exmp}
\begin{exmp}[The space $\PR$]\label{ex:The_space_PR}  We equip $\PR$ with the stochastic order and the topology of weak convergence. Recall that for $F, G \in \PR$ we say that \emph{$F$ is stochastically smaller or equal than $G$}, for short $F \st G$, if $F(x)\geq G(x)$ for all $x \in \R$. Recall further that the topology of weak convergence can be metrisized by the \emph{L\'evy metric}
\begin{equation*}
d(F,G)= \inf \hsp \big\{\epsilon > 0 \mid G(x-\epsilon)-\epsilon \leq F(x) \leq G(x+\epsilon)+\epsilon \textrm{ for all }x \in \R\big\}, \quad F,G \in \PR.
\end{equation*}
\end{exmp}

Let $(\XX', d', \preceq')$ be a further nonempty ordered metric space and let $\M \subseteq 2^\XX, \M' \subseteq 2^{\XX'}$ be $\s$-lattices. We call a function $f: \XX \to \XX'$ \emph{isotonic} if $x \preceq y$ implies $f(x) \preceq' f(y)$, and $\M-\M'$-\emph{measurable} if $f^{-1}(B) \in \M$ for all $B\in \M'$. It is straightforward to check that $f: \XX \to \XX'$ is increasing if and only if it is $\U-\U'$-measurable, where $\U$ and $\U'$ denote the family of all upper sets in $\XX$ and $\XX'$ respectively. This key observation explains the close relation between $\s$-lattices and isotonicity. %XXX If we want to extend Proposition \ref{prop:factorization_of_random_variables} for random elements, then we aim for a Borel measurable and increasing function $f: \XX \to \R$. 
Since $f: \XX \to \R$ is Borel measurable and increasing if and only if $f^{-1}(B) \in \BX$ for all $B \in \BR$ and $\{f > a\} \in \U$ for all $a \in \R$,  we propose the following definition.
 
	\begin{defn} \label{def:measurability_and_generated_sigma_lattices_for_random_elements} Let $\xi$ be a random element in $(\XX, d, \preceq)$ and let $\A \subseteq \F$ be a $\s$-lattice. 
		\begin{enumerate}
			\item $\xi$ is called \emph{$\A$-measurable} if $\xi^{-1}(B) \in \A$ for all $B \in \BX \cap \U$. 
			\item We define the \emph{$\s$-lattice generated by $\xi$} as the family of all preimages of measurable upper sets in $\XX$ under $\xi$ and denote it by $\LL(\xi)$, that is
			\begin{equation*}
				\LL(\xi)= \big\{\xi^{-1}(B)\mid B \in \BX \cap \U \big\}.
			\end{equation*}
		\end{enumerate}
	\end{defn}
Clearly, $\BX \cap \U$ as the intersection of a $\s$-algebra with a $\s$-lattice is itself a $\s$-lattice. Thus saying that a random element $\xi$ is $\A$-measurable is short for saying that $\xi$ is $\A-(\BX \cap \U)$-measurable.
Moreover, $\xi$ is $\A$-measurable if and only if $\LL(\xi) \subseteq \A$.

Definition \ref{def:measurability_and_generated_sigma_lattices_for_random_elements} is consistent with the classical definition, where we call a random variable $X$ $\A$-measurable if $\{X>a\}\in \A$ for all $a \in \R$. As it is usual convention to call a random variable $X$ measurable with respect to the $\s$-algebra $\F$ if $X$ is $\F-\BR$-measurable, we say that $X$ is $\A$-measurable to abbreviate that $X$ is $\A-\A_\R$-measurable, where $\A_\R$ denotes the standard lattice on $\R$, that is the $\s$-lattice generated by all open upper sets in $\R$, i.e.\ $\A_\R= \LL(\EE)$ for $\EE= \{(a, \infty)\mid a \in \R\}$. Since any upper set in $\R$ is of the form $(a,\infty)$ or $[a,\infty)$ for $a\in \R$, it follows directly that $\A_\R$ is equal to the $\s$-lattice of all measurable upper sets in $\R$. Finally, it may be verified that the random variable $X$ is $\A_\R$-measurable if and only if $X^{-1}(B)\in \A$ for any $B\in \EE$. The smallest $\s$-lattice $\mathcal{M} \subseteq \F$ such that $X$ is $\mathcal{M}$-measurable is $\LL(X)$, see for example \citet[p.\ 115]{Chu-in}. Equivalently, it is the smallest $\s$-lattice containing $\{X>a\}$ for all $a \in \R$, and it is not difficult to show that $\LL(X)$ basically only consists of its generator, that is,
\begin{equation}\label{eq:generated_sigma_algebra_for_random_variables}
    \LL(X) =  \big\{\{X>a\} \mid a \in \R \big\} \cup  \big\{\{X\geq a\} \mid a \in \R \big\}
\end{equation}
\citep[Lemma 5.2]{Masterthesis}. 

Under some regularity conditions on $(\XX, d, \preceq)$, which are explained in Appendix \ref{appendix:ordered_metric_spaces}, we can show the following result.
	
	\begin{lem}\label{lem:sigma_algebras_generated_by_generated_sigma_lattice}
		Consider a random element $\xi$ in an ordered metric space $(\XX, d, \preceq)$, which satisfies Assumption \ref{assump:ordered_metric_space}. Then
		\begin{equation*}
			\s(\xi)= \s\big(\LL(\xi)\big).
		\end{equation*}
		In particular any $\LL(\xi)$-measurable random variable $Z$ is $\s(\xi)$-measurable.
	\end{lem}
%It is important to note that the converse is not true, that is a $\s(\xi)$-measurable random variable $Z$ does not necessarily have to be $\LL(\xi)$-measurable. 
The ordered metric spaces $\R^p$ and $\PR$ given in Examples \ref{ex:Euclidean_space} and \ref{ex:The_space_PR}, respectively, satisfy Assumption \ref{assump:ordered_metric_space}. Therefore, Lemma \ref{lem:sigma_algebras_generated_by_generated_sigma_lattice} implies that $\s(\X)=\s(\LL(\X))$ for any random vector $\X\in \R^p$ and $\s(F)=\s(\LL(F))$ for any random cdf $F\in \PR$.

\subsection{Factorization and connection to the classical conditional law}\label{Subsec:Factorization}

Consider two random variables $X$ and $Y$. %Recall that $X$ is called $\A$-measurable if $\{X>a\} \in \A$ for all $a \in \R$. Hence, i
The classical factorization result states that a random variable $Y$ is $\s(X)$-measurable if and only if there exists a measurable function $f:\R\to \R$ such that $Y=f(X)$ \citep[Lemma 11.7]{Bauer_MI}. %In particular, if $Y \in L_1$, there exists a measurable function $f:\R\to \R$ such that $\E(Y\mid X)=f(X)$. %The following proposition is shown in \cite{Chu-in}. 
If we consider $\s$-lattices instead of $\s$-algebras, the factorization result remains valid with the additional requirement that the factorizing function $f$ needs to be isotonic \citep{Chu-in}. 
%\begin{prop}\label{prop:factorization_of_random_variables}
	%	\begin{equation*}
	%		\E\big( Y \hsp | \hsp \LL(X)\big)= f(X).
	%	\end{equation*}
We generalize this result to random elements $\xi$ with values in an ordered metric space.

 \begin{prop} \label{prop:factorization_for_random_elements}
		Let $\xi$ be a random element in some ordered metric space $(\XX,d,\preceq)$ and let $Z$ be a random variable. Then $Z$ is $\LL(\xi)$-measurable if and only if there exists an increasing Borel measurable function  $f: \XX \to \R$ such that $Z=f(\xi)$. In particular, for any integrable random variable $Y$, there exists an increasing Borel measurable function $f: \XX \to \R$ such that
		\begin{equation*}
			\E( Y \hsp | \hsp \LL(\xi))= f(\xi).
		\end{equation*}
	\end{prop}
	
Let $Y$ be a random variable and $\xi \in \XX$ be a random element with distribution $P_\xi$. Recall that the conditional law of $Y$ given $\xi$, which we denote by $P_{Y \mid \xi}$, is a Markov kernel from $(\Omega, \F)$ to $(\R, \BR)$ such that $P_{Y \mid \xi}(\cdot, B)$ is a version of $\prob(Y \in B \mid \xi)$ for any $B\in \BR$.
%Next we establish the connection between the conditional law of $Y$ given $\xi$, for short $P_{Y \mid \xi}$, and the isotonic conditional distribution $P_{Y \mid \LL(\xi)}$ for a random variable $Y$ and a random element $\xi \in \XX$ with distribution $P_\xi$. \\
So for any $y\in \Q$, there exits a measurable function $g_y: \XX \to \R$ such that $g_y(\xi)= \prob(Y >y \mid \xi)$ almost surely \cite[Lemma 11.7]{Bauer_MI} and hence
there exists a nullset $N\in \F$ such that $g_y(\xi(\omega))= P_{Y \mid \xi}(\omega, (y,\infty))$ for all $y\in \Q$ and for all $\omega \in \Omega \setminus N$. Thus for $P_\xi$-almost all $x\in \XX$, the family $(g_y)_{y \in \Q}$ characterizes a deterministic law $P_{Y \mid \xi=x}$, which we call \emph{conditional law of $Y$ given $\xi=x$}, by $P_{Y \mid \xi=x}((y,\infty))=g_y(x)$ for all $y\in \Q$ and imposing right-continuity.
	%Consider a further random variable $X$, and recall that by the classical factorization result \cite[Lemma 11.7]{Bauer2} there exists a measurable map $g: \R \to \R$ such that $\E(Y \mid X)=g(X)$ almost surely. So for $x \in \R$ we can define the conditional expectation of $Y$ given $X=x$, for short $\E(Y \mid X=x)$, as the real number $g(x)$. The conditional expectation $\E(Y \mid X=x)$ corresponds with the expectation of $Y$ with respect to the conditional probability measure $\prob( \cdot \mid X=x)$ if the event $\{X=x\}$ has positive probability, but is defined as well for all other $x\in \R$. \\
	% Since  Thus $P_{Y \mid \xi=x}$ is a fixed law such that $P_{Y \mid \xi=x}(B)= g_B(x)$ for all $B \in \BR$. 
 
 The following theorem shows that the conditional law of $Y$ given $\xi$ corresponds to the isotonic conditional law of $Y$ given $\xi$ if, and only if, the conditional law is isotonic in $\xi$, where equality between the Markov kernels $P_{Y \mid \xi}$ and $P_{Y \mid \LL(\xi)}$ is understood as usual in the sense that $P_{Y \mid \xi}(\cdot, B)=P_{Y \mid \LL(\xi)}(\cdot, B)$ almost surely for all $B \in \BR$. 
	
	\begin{thm}\label{thm:connection_to_classical_cond_laws} Let $Y$ be a random variable and let $\xi$ be a random element in an ordered metric space $(\XX,d,\preceq)$ that satisfies Assumption \ref{assump:ordered_metric_space}. Then $P_{Y \mid \xi}=P_{Y \mid \LL(\xi)}$ if and only if the conditional law of $Y$ given $\xi$ is almost surely isotonic in $\xi$, that is for $P_\xi$-almost all $x, x' \in \XX$ with $x \preceq x'$ it holds that $P_{Y \mid \xi = x} \st P_{Y \mid \xi = x'}$. 
	\end{thm}

For any $\s$-lattice $\A \subseteq \F$ and random variable $Y$, we have shown existence and uniqueness of the ICL $P_{Y \mid \A}$ and its connection to isotonicity. In the next section, we show that $P_{Y \mid \A}$ equivalently emerges as a minimizer of some expected scoring rule (Proposition \ref{prop:crps_characterization_of_ICL}) and that $P_{Y \mid \A}$ is even a universal minimizer in the sense that we may allow for a large class of scoring rules (Theorem \ref{thm:universality_of_ICL}).

\section{Universality of ICL}\label{Sec:Universality_of_ICL}

%In Theorem \ref{thm:universality_of_ICL} we will show that for any $z\in \R$ and any $\alpha \in (0,1)$, the random variables $P_{Y \mid \A}(z)$ and $P_{Y \mid \A}^{-1}(\alpha)$ are simultaneously optimal with respect to large classes of scoring functions. Before we state these results more precisely, we recall some basics notions and definitions of the forecast evaluation literature. 

\subsection{Identifiable functionals, scoring functions and scoring rules}\label{Subsec:Forecast_Evaluation}
%A \emph{statistical functional} $T$ is a function from some class $\P \subseteq \PR$ into $2^{\Bar{\R}}$. A functional $T$ is called \emph{of singleton type} if $T(P)$ is a singleton whenever $P$ is not the null measure. Otherwise, $T$ is called a \emph{functional of interval type}. 
%We consider \emph{identifiable functionals} as studied by \cite{Ehm_et_al} and \cite{Optimal}. 
In this section, we introduce identifiable functionals, scoring functions and scoring rules in order to show some universality properties of ICL. 

Following \cite{Optimal}, we call $V: \R \times \R \to \R$ an \emph{identification function}, if $V(\cdot,y)$ is increasing and left-continuous for all $y \in \R$. Similar to the classical procedure in $M$-estimation, an identification function weights negative values in the case of underprediction against positive values in the case of overprediction \citep{Tilmann_Johannes_Calibration}. %Any identification function induces a functional which maps to the possibly set-valued argument at which an associated expectation changes sign: 
For an identification function $V$, we define the \emph{functional $T$ induced by $V$} as $T(F) = [T_F^-, T_F^+] \subseteq \bar{\R} =\R \cup \{\pm \infty\}$, for any $F\in \PR$, where the lower and upper bounds are given by
\[
T_F^- = \sup \left\{x \MMid \int_{-\infty}^{\infty} V(x, y) \diff F(y) < 0\right\} \eqand
T_F^+ = \inf \left\{x \MMid \int_{-\infty}^{\infty} V(x, y) \diff F(y) > 0\right\}.
\]
A functional $T$ is called \emph{of singleton type} if $T(F)$ is a singleton, and otherwise, $T$ is called a \emph{functional of interval type}. Relevant key examples for this paper are the mean functional with identification function  $V(x,y)=x-y$, and the quantile functional at level $\alpha\in (0,1)$ with identification function 
\begin{equation}\label{eq:identification_function_quantile}
    V^\alpha(x,y)=\one_{\{x>y\}}-\alpha.
\end{equation} 
Clearly, the mean functional is of singleton type but, in general, the quantile functional is of interval type. One may instead consider the singleton-valued \emph{lower quantile} $T_F^-$ or \emph{upper quantile} $T_F^+$. For any distribution $F$, almost all quantiles are uniquely defined, i.e.\ the family of $\alpha \in (0,1)$, for which the set of $\alpha$-quantiles forms an interval, has Lebesgue measure zero. %Moreover, the nonuniqueness of quantiles is arguably irrelevant from a statistical point of view :references?. For these reasons, 
We will mainly consider the lower-quantile $T_F^-$, which corresponds to the quantile function $F^{-1}$ evaluated at $\alpha \in (0,1)$, where $F^{-1}(\alpha) = \inf \{z \in \R \mid F(z) \geq \alpha\}$, $\alpha \in (0,1)$.

Point-valued forecasts for functionals of an unknown quantity should be compared using consistent scoring functions \citep{Gneiting_2011}. A measurable map $S : \mathbb{R}\times \mathbb{R}\rightarrow \bar{\R}$ is a \emph{consistent scoring function for the functional $T$ relative to the class $\mathcal{P} \subseteq \PR$} if for all $x \in \R$, $F \in \mathcal{P}$, the integral $\int S(x,y)\diff F(y)$ exists and $\int S(t,y)\diff F(y) \leq \int S(x,y) \diff F(y)$ for all $t\in T(F)$. The \emph{Brier score} (or \emph{quadratic score}) $\BS(x,y)=(x-y)^2$, $x,y \in \R$, and the \emph{quantile score}
\begin{align*}
    \QS_\alpha(x,y) &=(\one_{\{y \leq x\}}-\alpha)(x-y), \quad x,y \in \R,
   % &= \int_\R S_{\alpha,\eta}^Q(x,y) \diff \eta
\end{align*}
are main examples of consistent scoring functions for the mean functional, and for the quantile at level $\alpha \in (0,1)$, respectively. 

We adapt the methods of \cite{Optimal} that rely on the mixture or Choquet representation of consistent scoring functions introduced by \cite{Ehm_et_al} for the quantile and expectile functionals. Under some standard conditions, for any $\alpha \in (0,1)$, consistent scoring functions for the $\alpha$-quantile have an integral representation $\int_\R S_{\alpha,\eta}^Q(x,y) \diff H(\eta)$, with respect to the elementary $\alpha$-quantile loss functions
\begin{equation}\label{eq:elementary_loss_function_for_the_quantile}
    S_{\alpha,\eta}^Q(x,y) = (1-\alpha)\one_{\{y < \eta \le x\}} + \alpha\one_{\{x < \eta \le y\}},\quad \eta \in \R,
\end{equation}
where $H$ is a nonnegative $\s$-finite measure on $\R$. The quantile score $\QS_\alpha(x,y)$ arises when $H$ is the Lebesgue measure. Similarly, any consistent scoring function for the mean is a Bregman scoring function and can be written as $\int S_{\eta}^E(x,y) \diff H(\eta)$, where 
\begin{equation}\label{eq:elementary_loss_function_for_the_mean}
    S_{\eta}^E(x,y)=(\eta-y)\one_{\{y < \eta \le x\}} + (y-\eta)\one_{\{x < \eta \le y\}}, \quad  \eta \in \R,
\end{equation}
and $\BS(x,y)$ arises by taking $H$ to be the Lebesgue measure. 

%and for any $F\in \PR$ and $y,z \in \R$, $\BS(F(z),\one\{y\leq z\}) = \int_0^1 S_{z,\eta}^P(F,y) \diff \eta$, where

%For an identification function $V$ for the functional $T$, we consider the family of \emph{elementary loss functions} that are consistent for the functional $T$ 
%\begin{equation}\label{eq:elementary_scoring_function}
%    S_\eta(x,y) = \one\{\eta \leq x\}V(\eta,y) + |V(\eta,y)| =\left(\one\{\eta \leq x\} - \one\{\eta \leq y\}\right)V(\eta,y), \quad \eta \in \R. 
%\end{equation}
%If $T$ is the expectation or the quantile functional at level $\alpha \in (0,1)$, then the class
%\begin{equation}\label{eq:class_of_consistent_scoring_functions}
%    \mathscr{S}= \left\{\int_\R S_\eta(x,y) \diff H (\eta): H \textrm{ is a nonnegative measure on }\R\right\}
%\end{equation}
%comprises all consistent loss functions for $T$ subject to standard conditions, see \cite{Optimal}. This important result will be used repeatedly. In particular, 

Probabilistic forecasts, that provide full predictive distributions over the possible values of the future outcome should be compared and evaluated using proper scoring rules \citep{Gneiting_Raftery_2007}. A \emph{proper scoring rule} is a function $\myS: \mathcal{P}\times \R \mapsto \bar{\R}$ for some class $\mathcal{P}\subseteq \PR$ such that $\myS(F, \cdot)$ is measurable for any $F\in \mathcal{P}$, the integral $\int\myS(G,y) F(\diff y)$ exists, and
\begin{equation*}
    \int\myS(F,y) \diff F( y) \leq \int\myS(G,y) \diff F( y)
\end{equation*}
for all $F,G \in \mathcal{P}$.  A key example is the \emph{continuous ranked probability 
score} (CRPS; \cite{Matheson1976}), which is defined for all Borel probability measures, equivalently given by
\begin{equation*}
    \crps(F,y)=\int \BS(F(z),\one_{\{y\leq z\}}) \diff z= 2 \int_0^1 \QS_\alpha (F^{-1}(\alpha),y) \diff \alpha.
\end{equation*}

%Recall that equivalently we may define 
%\begin{equation}\label{eq:def_generalized_inverse_with_sup}
%    F^{-1}(\alpha) = \sup \{z \in \R \mid F(z) < \alpha\}, \quad \alpha \in [0,1].
%\end{equation}
%\cite{Gneiting_Raftery_2007} show that the CRPS can equivalently be expressed as
%\begin{equation}
%    \crps(F,y) = \E \left|X-y\right|-\E \left|X-X'\right|,
%\end{equation}
%where $X$ and $X'$ are independent copies with distribution $F$, or as

%Recall that for any identifiable functional we consider the elementary loss function
%\begin{align}
%	S_\eta(x,y) &= \one\{\eta \leq x\}V(\eta,y) + |V(\eta,y)|\\
% & =\left(\one\{\eta \leq x\} - \one\{\eta \leq y\}\right)V(\eta,y).
%\end{align}

%In summary, we have the representations
%\begin{align}
%    \crps(F,y)&=\int_\R \int_0^1 S_{x,c}^P(F,y) \diff c \diff x = 2 \int_0^1  \int_\R S_{\alpha,\eta}^Q(x,y) \diff \eta \diff \alpha.
%\end{align}

\subsection{ICL as a scoring rule minimizer}\label{Subsec:ICL_as_a_CRPS_minimizer}

We denote by $\HH$ the family of all Markov kernels from $(\Omega, \F)$ into $(\R, \BR)$. Any Markov kernel $Q \in \HH$ can equivalently be considered as a random element in $\PR$. %, see Appendix \ref{appendix:Markov_Kernels}. 
If it is clear from the context that $Q$ is random we write $Q(x)$ or $Q(\cdot, x)$ to abbreviate $ Q(\cdot, (-\infty, x])$ and $\bar{Q}(x)$ or $\Bar{Q}(\cdot, x)$ to abbreviate $Q(\cdot, (x,\infty))$ for $x\in \R$. The next Proposition shows that ICL equivalently results as a CRPS minimizer.
\begin{prop} \label{prop:crps_characterization_of_ICL}
For a random variable $Y$ and a $\s$-lattice $\A\subseteq \F$, the ICL $P_{Y \mid \A}$ minimizes
\begin{equation}\label{eq:crps_minimizing_cond_of_ICL}
\E \crps (G,Y)
\end{equation}
over all random distributions $G$ in the class 
\begin{equation*}\label{eq:class_G_A}
    \GG_\A=\{G \in \HH \mid G(\cdot, (z, \infty))\textrm{ is }\A\textrm{-measurable for all }z \in \R \}.
\end{equation*}
The ICL $P_{Y\mid \A}$ is the unique minimizer in the sense that $P_{Y \mid \A}(\cdot, B)= G(\cdot, B)$ almost surely for all $B \in \BR$, for any $G \in \GG_\A$ that minimizes \eqref{eq:crps_minimizing_cond_of_ICL} as well. 
\end{prop}
\begin{rem}[The class $\GG_\A$]\label{rem:class_GG_A} The class $\GG_\A$ can equivalently be defined as
\begin{equation*}
    \GG_\A = \{G \in \HH \mid G^{-1}(\alpha) \textrm{ is }\A \textrm{-measurable for all }\alpha \in (0,1)\}.
\end{equation*} This may be seen by the equality
\begin{equation*}
    \Omega \setminus \{G(\cdot, (z,\infty))>1-\alpha\} = \{G(z)\geq \alpha\}= \{z \geq G^{-1}(\alpha)\}= \Omega \setminus \{G^{-1}(\alpha)>z\},
\end{equation*}
for any $z\in \R$ and $\alpha \in (0,1)$ and where the second equality follows by the fact that $G(z)\geq \alpha$ if and only if $z\geq G^{-1}(\alpha)$, see e.g.\ \cite{Generalized_inverses}. 

Suppose now that $\A$ is generated by a random element $\xi$ in some ordered metric space $(\XX,d,\preceq)$ that satisfies Assumption \ref{assump:ordered_metric_space} and let $P_\xi$ denote the distribution of $\xi$. For any $G\in \HH$, $G_{\xi}$ and $G_{\LL(\xi)}$ are Markov kernels such that $G_{\xi}(\cdot, B)=\E(G(\cdot, B)\mid \xi)$ almost surely for any $B\in \BR$ and $G_{\LL(\xi)}(\cdot, (z,\infty))=\E(G(\cdot, (z,\infty))\mid \LL(\xi))$ almost surely for any $z\in \R$. Clearly, for any $\s$-lattice $\A \subseteq \F$, we have $G=G_\A$ if and only if $G(\cdot, (z,\infty))$ is $\A$-measurable. By Proposition \ref{prop:factorization_for_random_elements}, the class $\GG_\A$ may alternatively be characterized as
\begin{equation*}
\GG_\A= \{G \in \mathscr{H}\mid  G_{\xi=x} \st G_{\xi=x'} \textrm{ for $P_\xi$-almost all }x,x'\in \XX \textrm{ with }x \preceq x'  \}.
\end{equation*}

Indeed, assume that for $G\in \HH$ it holds that $G(\cdot, (z,\infty))$ is $\LL(\xi)$-measurable for any $z\in \R$. By Proposition \ref{prop:factorization_for_random_elements}, for any $z\in \R$, there exists an increasing Borel measurable function $g_z$ such that $G(\cdot, (z,\infty))=g_z(\xi)$ almost surely. Recall that for $P_\xi$-almost all $x\in \XX$, $G_{\xi=x}$ denotes a deterministic law which is characterized by $G_{\xi=x}((z,\infty))=g_z(x)$ for all $z\in \Q$ and imposing right-continuity. Since $g_z$ is isotonic for any $z\in \R$, it follows directly that $G_{\xi=x} \st G_{\xi=x'}$ for $P_\xi$-almost all $x,x'\in \XX$ with $x\preceq x'$. 

Reversely, assume that for $G\in \HH$ it holds that $G_{\xi=x} \st G_{\xi=x'}$ for $P_\xi$-almost all $x,x'\in \XX$ with $x\preceq x'$. Then $G_{\xi}=G_{\LL(\xi)}$ by Theorem \ref{thm:connection_to_classical_cond_laws}. That is for any $z\in \R$, $G_\xi(\cdot, (z,\infty))$ is a regular version of $\E(G(\cdot,(z,\infty)\mid \LL(\xi))$ and hence $\LL(\xi)$-measurable. 

\end{rem}
Proposition \ref{prop:crps_characterization_of_ICL} allows to show that ICL generalizes isotonic distributional regression (IDR) introduced by \citet{IDR}. IDR is a nonparametric distributional regression method that uses the shape constraint of isotonicity between the covariates and the responses: Consider
%IDR is a powerful tool to estimate conditional laws under the shape constraint of isotonicity. IDR is fully automatic in the sense that we do not have to bin or choose any tuning parameter (eg bandwidth in density estimation). 
training data consisting of covariates $\bx=(x_1, \dots, x_n)$ in a partially ordered set $(\XX, \preceq)$ and real valued responses $\by = (y_1, \dots, y_n) \in \R^n$ for $n\in \N$. \citet{IDR} show that there exists a unique minimizer of the criterion
\begin{equation}\label{eq:IDR_minimizing_criterion}
\frac{1}{n} \sum_{i=1}^{n}\crps(G_i,y_i)
\end{equation}
over all vectors of distributions $(G_1,\dots,G_n)$ with $G_i \st G_j$ if $x_i \preceq x_j$ for $i,j=1, \dots, n$ and call this minimizer \emph{Isotonic Distributional Regression (IDR) of $\by$ given $\bx$}. %Moreover they show that IDR is not only optimal with respect to the CRPS in the empirical expectation in \eqref{eq:IDR_minimizing_criterion} but with respect to a broad class of proper scoring rules that include all relevant choices in the extant literature, see \citet[Theorem 2.2]{IDR}.\\
The key observation in showing that ICL generalizes IDR is that any covariate vector $\bx \in \XX^n$ induces a $\s$-lattice $\A$ on the set of indices $\{1, \dots, n\}$ through the corresponding order relations, that is $A\in \A$ for $A\subseteq \{1, \dots, n\}$ if and only if $i\in A$ and $x_i\preceq x_j$ implies $j\in A$. Any vector of distributions $(G_1, \dots, G_n)$ may be considered as a random cdf $G$ defined on the discrete probability space $\{1, \dots, n\}$ and $G \in \GG_\A$ if and only $G_i \st G_j$ if $x_i \preceq x_j$ for any $i,j \in \{1, \dots, n\}$. 
Hence, we exactly obtain the minimizing criterion \eqref{eq:IDR_minimizing_criterion} if we consider the expectation with respect to the empirical measure of the observations in \eqref{eq:crps_minimizing_cond_of_ICL}, and thus ICL generalizes IDR by Proposition \ref{prop:crps_characterization_of_ICL}. Measurability poses no problem in the finite sample case. For this reason \cite{IDR} can simply consider partially ordered spaces but we have to study ordered metric spaces. 

%\subsection{Universality of ICL}
Analogously to Theorem 2.2 of \cite{IDR}, the next theorem shows that $P_{Y \mid \A}$ is not only optimal with respect to the CRPS but with respect to a broad class of proper scoring rules, and that for any $z\in \R$ and $\alpha \in (0,1)$, the random variables $P_{Y \mid \A}(z)$ and $P_{Y \mid \A}^{-1}(\alpha)$ are optimal with respect to all consistent scoring functions for the mean and the $\alpha$-quantile, respectively.
\begin{thm}[Universality of ICL]\label{thm:universality_of_ICL}
For a random variable $Y$ and a $\s$-lattice $\A\subseteq \F$, the ICL $P_{Y \mid \A}$ satisfies the following properties. 
\begin{enumerate}
\item The ICL $P_{Y \mid \A}$ minimizes $\E \myS(G,Y)$ over all $G \in \GG_\A$ uniquely in the sense of Proposition \ref{prop:crps_characterization_of_ICL} under any scoring rule $\myS$ of the form
\begin{equation*}
\myS(F,y) = \int_{\R \times \R} S_{\eta}^E(F(z),\one_{\{y \le z\}}) \diff H(z,\eta),
\end{equation*}
where  $S_{\eta}^E$ is defined at \eqref{eq:elementary_loss_function_for_the_mean} and $H$ is a $\s$-finite Borel measure on $\R \times \R$. Moreover, for every threshold value $z\in \R$, $P_{Y \mid \A}(\cdot, (z,\infty))$ is a minimizer of
\begin{equation*}
    \E S(X, \one_{\{Y > z\}})
\end{equation*}
over all ${\A}$-measurable random variables $X$, where $S$ is any consistent scoring function for the mean, that is, $S(x,y) = \int S_\eta^E(x,y)\diff \tilde{H}(\eta)$ for some $\s$-finite Borel measure $\tilde{H}$ on $\R$.
\item Suppose that the random vector $(P_{Y\mid \A}^{-1}(\alpha),Y)$ satisfies Assumption \ref{assump:joint_distribution} for any $\alpha \in (0,1)$. Then, the ICL $P_{Y \mid \A}$ minimizes $\E \myS(G,Y)$ over all $G \in \GG_\A$ uniquely in the sense of Proposition \ref{prop:crps_characterization_of_ICL} under any scoring rule $\myS$ of the form
\begin{equation*}
\myS(F,y) = \int_{(0,1)\times \R} S_{\alpha,\eta}^Q(F^{-1}(\alpha),y) \diff H(\alpha, \eta)
\end{equation*}
 where $S_{\alpha,\eta}^Q$ is defined at \eqref{eq:elementary_loss_function_for_the_quantile}, and $H$ is a $\s$-finite Borel measure on $(0,1) \times \R$. Moreover, for every quantile level $\alpha \in (0,1)$, $P_{Y \mid \A}^{-1}(\alpha)$ minimizes 
\begin{equation*}
\E S(X,Y)
\end{equation*}
over all $\A$-measurable random variables $X$, where $S$ is any consistent scoring function for the $\alpha$-quantile, that is, $S(x,y) = \int S_{\alpha,\eta}^Q(x,y)\diff \tilde{H}(\eta)$ for a $\s$-finite Borel measure $\tilde{H}$ on $\R$.
 \end{enumerate}
\end{thm}

We strongly believe that the second part of the theorem also holds without Assumption \ref{assump:joint_distribution}. However, we are currently not able to give a proof. We discuss the assumption in Section \ref{Sec:Conditional_Functionals}.

%\cite{IDR} discuss that Theorem \ref{thm:universality_of_ICL} compromises all scoring rules and scoring functions for the mean and quantile functional which are relevant in the extant literature and hence we refer to Theorem \ref{thm:universality_of_ICL} as the universality theorem. 

For the proof of Theorem \ref{thm:universality_of_ICL}, we use the fact that ICL is threshold and quantile calibrated as we will establish in the next section. Furthermore, we introduce a way to view conditional functionals as minimizers of consistent scoring functions in Section \ref{Sec:Conditional_Functionals} by adapting the methods of \cite{Optimal}, and this will allow us to conclude.

\section{Calibration properties of ICL}\label{Sec:Calibration}
%Calibration is a key requirement in statistical forecasting. Roughly speaking, a forecasts is \emph{calibrated} or \emph{reliable} if it provides a plausible probabilistic explanation of the observations \citep{Gneiting_Katzfuss_2014}. Numerous different notions of calibration are found in the literature, see e.g.\ \cite{Dawid1984} and \cite{Diebold1998} for probabilistic calibration, \cite{Strahl_Ziegel_2017} and \cite{Arnold_Henzi_Ziegel} for conditional and sequential notions of calibration, \cite{Thorarinsdottir_et_al_2016} and \cite{Ziegel_Gneiting_2014} for multivariate notions of calibration and \cite{Guo_et_al_2017} and \cite{Gupta_et_al_2022} for applications in the machine learning community. A recent extensive study of the logical dependencies between the main notions of calibration is given by \cite{Tilmann_Johannes_Calibration}. In this section, we introduce the new notion of isotonic calibration and study the logical dependencies with respect to the three familiar types of auto-, threshold- and quantile-calibration. Moreover, we show that ICL is threshold- and quantile calibrated, which will particularly be used to show Theorem \ref{thm:universality_of_ICL}. 

Probabilistic predictions are called calibrated if they are statistically consistent with the observed outcomes. 
Consider a random variable $Y$ and a (random) predictive distribution for $Y$, denoted by $F$. The random forecast $F$ is called \emph{auto-calibrated} \citep{Tsyplakov2013}, if the conditional law of $Y$ given $F$ equals $F$, that is $P_{Y\mid F}=F$, or equivalently
\begin{equation}\label{defn_auto_calibration}
    F(y)= \prob(Y \leq y \mid F) \quad \textrm{almost surely for all } y \in \R.
\end{equation}
For any $y\in \R$, we may condition on $F(y)$ instead of $F$ in \eqref{defn_auto_calibration}, to obtain the weaker notion of threshold calibration, i.e.\ we call the forecast $F$ \emph{threshold calibrated} \citep{IDR}, if
\begin{equation*}\label{defn_threshold_calibration}
    F(y)= \prob(Y \leq y \mid F(y)) \quad \textrm{almost surely for all } y \in \R.
    \end{equation*}
%Essentially, for a threshold-calibrated forecast $F$, we can take $F(y)$ at face value for any $y\in \R$. For a statistical functional $T$, \cite{Tilmann_Johannes_Calibration} call the forecast $F$ (conditionally) $T$-calibrated if $T(F)$ almost surely equals $T$ applied to the conditional law of $Y$ given $T(F)$. Thus, threshold-calibration emerges as a particular instance of $T$-calibration if we let $T$ be the event probability $T(F)=F(y)$ for $y\in \R$. Analogously, we speak of quantile-calibration, if we let $T$ be the quantile-functional. 
We adapt the definition of \cite{Tilmann_Johannes_Calibration} slightly and call the forecast $F$ \emph{quantile calibrated}, if
\begin{equation}\label{defn_quantile_calibration}
    F^{-1}(\alpha)= q_\alpha(Y \mid F^{-1}(\alpha)) \quad \textrm{almost surely for all } \alpha \in (0,1), 
    \end{equation}
where for any $\alpha \in (0,1)$, $q_\alpha(Y \mid F^{-1}(\alpha))$ denotes the lower-$\alpha$-quantile of the conditional law of $Y$ given $F^{-1}(\alpha)$. Restricting to the lower quantile ensures that all random quantities in \eqref{defn_quantile_calibration} are real-valued.
%and hence we do not have to study generated $\s$-algebras or measurability for random sets. 
Since almost all quantiles of the posited distribution $F$ are point-valued, %in any case and set-valued quantiles are arguably irrelevant from a practical point of view,
%as argued in Section \ref{Subsec:Forecast_Evaluation}, 
the restriction to lower quantiles in the definition of quantile calibration is not limiting. For $\alpha \in (0,1)$, let $X=q_\alpha(Y \mid F^{-1}(\alpha))$. Then $\prob(Y<X \mid F^{-1}(\alpha)) \leq \alpha \leq \prob(Y\leq X \mid F^{-1}(\alpha))$ almost surely, see e.g.\ \cite{Tomkins}. Hence, for any quantile calibrated forecast $F$,
%For $\alpha \in (0,1)$, we have $\prob(Y<q_\alpha(Y \mid F^{-1}(\alpha)) \mid F^{-1}(\alpha)) \leq \alpha \leq \prob(Y\leq q_\alpha(Y \mid F^{-1}(\alpha)) \mid F^{-1}(\alpha))$ almost surely, see e.g.\ \cite{Tomkins}. Hence for any quantile calibrated forecast $F$, 
\begin{equation*}
    \prob(Y<F^{-1}(\alpha) \mid F^{-1}(\alpha)) \leq \alpha \leq \prob(Y\leq F^{-1}(\alpha) \mid F^{-1}(\alpha))\quad \textrm{almost surely for all } \alpha \in (0,1).
\end{equation*}
\citet{Tilmann_Johannes_Calibration} provide a recent comprehensive analysis of the connections and implications between different notions of calibration found in the literature.

Let $\A \subseteq \F$ be a $\s$-lattice. Clearly, the ICL $P_{Y \mid \A}$ may be interpreted as a predictive distribution for $Y$. In fact, $P_{Y \mid \A}$ is the theoretically optimal distributional prediction for $Y$ given all information contained in $\A$. The following proposition shows that ICL is threshold- and quantile-calibrated.
\begin{prop}\label{prop:ICL_is_threshold_calibrated}
    For any $\s$-lattice $\A \subseteq \F$, the ICL $P_{Y \mid \A}$ is threshold calibrated, that is,
    \begin{equation*}
        P_{Y \mid \A}(y) = \prob(Y \leq y \mid P_{Y \mid \A}(y))\quad \textrm{almost surely for all } y \in \R.
    \end{equation*}
\end{prop}
\begin{prop}\label{prop:ICL_is_quantile_calibrated}
    Let  $\A \subseteq \F$ be a $\s$-lattice, and suppose that the random vector $(P_{Y\mid \A}^{-1}(\alpha),Y)$ satisfies Assumption \ref{assump:joint_distribution} for any $\alpha \in (0,1)$. Then, the ICL $P_{Y \mid \A}$ is quantile calibrated, that is
    \begin{equation*}
        P_{Y \mid \A}^{-1}(\alpha) = q_\alpha(Y \mid P_{Y \mid \A}^{-1}(\alpha))\quad \textrm{almost surely for all } \alpha \in (0,1).
    \end{equation*}
\end{prop}
Proposition \ref{prop:ICL_is_threshold_calibrated} may be shown directly using the properties of conditional expectations and isotonic conditional laws. The analogous claim for quantiles is more involved and needs a further analysis of condition quantiles. For this reason, we will prove Proposition \ref{prop:ICL_is_quantile_calibrated} in Section \ref{Subsec:cond_quantiles}. We strongly believe that Proposition \ref{prop:ICL_is_quantile_calibrated} also holds without Assumption \ref{assump:joint_distribution} but we can currently not give a fully general proof, see Section \ref{Subsec:cond_quantiles} for a detailed discussion.

Threshold- and quantile-calibration of ICL together imply that ICL satisfies a further important calibration property, which is known as \emph{probabilistic calibration} \citep{GneitingRanjan2013}. The argument, which holds for any threshold- and quantile-calibrated $F$, is due to Alexander Henzi and given in Appendix \ref{appendix:PIT_calibration}. 

%Recall that we call $F$ auto-calibrated if $F$ coincides with the conditional law of $Y$ given $F$. 
The following definition is natural. Proposition \ref{prop:implications_calibration} is illustrated in Figure \ref{figure1}.
\begin{defn}\label{def:isotonic_calibration}
		We call a probabilistic forecast $F$ \emph{isotonically calibrated} if $P_{Y \mid \LL(F)}=F$.
	\end{defn}
%The following chain of implications shows that isotonic calibration is in between auto-calibration and threshold- or quantile-calibration, see Figure \ref{figure1}.
\begin{prop}\label{prop:implications_calibration}
    Auto-calibration implies isotonic calibration, and isotonic calibration implies threshold calibration and quantile calibration.
\end{prop}
\begin{figure}[H]
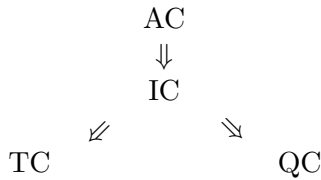

\centering
\begin{tabular}{ccccc}
 &  &  AC & & \\ 
&  &    $\Downarrow$ & & \\
&  &  IC & &\\
& $\SWarrow$    & & $\SEarrow$  & \\
TC &&&& QC
 \end{tabular}
\caption{The implications between auto-calibration (AC), isotonic calibration (IC), threshold calibration (TC) and quantile calibration (QC) as given in Proposition \ref{prop:implications_calibration}.}
\label{figure1}
\end{figure}
The following examples show that the reverse directions in Proposition \ref{prop:implications_calibration} do not hold in general. Hence, isotonic calibration is strictly weaker than auto-calibration and strictly stronger than threshold- or quantile calibration. 
\begin{exmp}
    To see that isotonic calibration does not imply auto-calibration, one can consider the equiprobable forecast $F$ in Figure 3 of \cite{Tilmann_Johannes_Calibration}, where we are given the two equiprobable realisations $F_1$ and $F_2$ of $F$ and the conditional law of $Y$ given $F=F_i$, for short $Q_i$, for $i=1,2$. Clearly, $F$ fails to be auto-calibrated. Since $F_2\st F_1$ but $Q_2(z)\leq Q_1(z)$ for $z\in [0,2]$, we have to pool $Q_1$ and $Q_2$ for $z\in [0,2]$ to obtain the isotonic conditional law of $Y$ given $F=F_i$, for short $G_i$. In summary $G_i(z)=(Q_1(z)+Q_2(z))/2$ for  $z\in [0,2]$ and $G_i(z)=Q_i(z)$ for $z\in [2,4]$, and hence $G_i=F_i$, that is, $F$ is isotonically calibrated.
\end{exmp}
\begin{exmp}
    Consider Example 2.14 (a) of \cite{Tilmann_Johannes_Calibration}, where the forecast $F$ may attain four different mixtures of uniform distributions. Here, $F$ is threshold- and quantile-calibrated (with respect to the lower quantile) but fails to be isotonically calibrated.
\end{exmp}

If the outcome is binary, that is, $Y\in \{0,1\}$, the probabilistic forecast $F$ takes the form of a predictive probability $X\in [0,1]$. In this case, auto-calibration serves as an universal notion of calibration \cite[Theorem 2.11]{GneitingRanjan2013}, and the next lemma implies that auto-calibration and isotonic calibration are then equivalent. The result connects conditional expectations given $\s$-lattices and conditional expectations given $\s$-algebras and is of its own interest since it is not limed to binary outcomes. 
\begin{lem}\label{lem:equivalence_auto_and_iso_calibration_for_random_variables}
    For any integrable random variables $X$ and $Y$, we have $\E(Y \mid X)=X$ if and only if $\E(Y \mid \LL(X))=X$. 
\end{lem}

\section{Conditional functionals given $\sigma$-lattices}\label{Sec:Conditional_Functionals}
\subsection{General construction}
In this section, we propose a general notion of conditional (identifiable) functionals building on the work of \cite{Optimal}. Although our main motivation is to consider conditional quantiles in order to show Theorem \ref{thm:universality_of_ICL} and Proposition \ref{prop:ICL_is_quantile_calibrated}, the considerations in this section are of independent interest. Moreover, the given recipe is perfectly applicable as well for $\s$-algebras instead of $\s$-lattices as illustrated in Remark \ref{rem:conditional_functionals_for_sigma_algebras}. 

Consider a random variable $Y$ and an identifiable functional $T$ with identification function $V$. For a $\s$-lattice $\A \subseteq \F$, our aim is to find an $\A$-measurable random variable $X$ that minimizes
\begin{equation}\label{eq:minimizing_criterion_conditional_functionals}
\E (S(X,Y)),
\end{equation}
over all $\A$-measurable random variables $X$ such that the expectation exists, where $S$ is a consistent scoring function for $T$ with representation $S(x,y) = \int S_\eta(x,y)\diff H(\eta)$ for some $\s$-finite measure $H$ on $\R$, and elementary loss functions
\begin{equation}\label{eq:Seta}
S_\eta(x,y) = V(\eta,y)\one_{\{y < \eta \le x\}} - V(\eta,y)\one_{\{x < \eta \le y\}}, \quad \eta \in \R.
\end{equation}
The minimizing criterion \eqref{eq:minimizing_criterion_conditional_functionals} can be motivated by the fact that the conditional expectation of any square integrable $Y$ given $\A$ minimizes $\E(X-Y)^2$ over all square integrable $\A$-measurable random variables $X$. Clearly, also conditional quantiles may be characterized as minimizers of some expected loss, see e.g.\ \cite{Armerin}. 

Throughout this section, we assume that $\E(V(\eta,Y))$ is finite for all $\eta \in \R$. If there exists an $\A$-measurable random variable $X$ that minimizes
\begin{equation}\label{eq:minimizing_criterion_conditional_functionals_elementary}
\E(S_\eta(X,Y))
\end{equation}
   over all $\A$-measurable random variables $X$ simultaneously for all $\eta \in \R$, then it follows that $X$ is also a solution for the minimizing problem in \eqref{eq:minimizing_criterion_conditional_functionals}. For a fixed $\eta \in \R$, minimizing equation \eqref{eq:minimizing_criterion_conditional_functionals_elementary} over the class of all $\A$-measurable random variables $X$ is equivalent to minimizing
	\begin{equation}\label{eq:Vmin}
 \int_{\{X \geq \eta\}} V(\eta, Y) \diff \prob
	\end{equation}
 over all $\A$-measurable random variables $X$. Since $X$ is $\A$-measurable, it follows that $\{X \geq \eta\} \in \A$ for all $\eta \in \R$. For $\eta \in \R$ and $A \in \A$, we hence consider
\begin{equation}\label{eq:def_s_A}
    s_A(\eta)= \E (\one_A V(\eta,Y)).
 \end{equation}
 Although our aim is to minimize $s_A(\eta)$ over all $A \in \A$ for any $\eta \in \R$, it will be useful to consider the same criterion but allow for general $A \in \F$. To point out this distinction, we define for $A \in \F$ \begin{equation*}
    v_A(\eta)= \E (\one_A V(\eta,Y)) .   %=\int V\big(\eta,Y(\omega) \big)\diff \nu_A(\omega),
 \end{equation*}
 %where $\nu_A(\cdot)= \prob(\cdot \cap A)$ is a nonnegative measure on $(\Omega, \F)$. We note that $\nu_A$ is the zero measure if and only if $A$ is a nullset. 
 That is $ s_A(\eta)= v_A(\eta)$ for $A \in \A$, while $ s_A(\eta)$ is not defined for $A \in \F \setminus \A$. For $\eta \in \R$, denote by $M(\eta)$ the family of all sets $A \in \A$ that minimize \eqref{eq:def_s_A}. The following lemma guarantees that $M(\eta)\neq \emptyset$ for all $\eta \in \R$. 
 \begin{lem}\label{lem:M_eta_is_non_empty}
     For $\eta \in \R$, suppose that $A_n \in \A$ for all $n \in \N$ such that $s_{A_n}(\eta) \to \inf_{A \in \A} s_A(\eta)$. Then $\liminf_{n\to \infty} A_n \in M(\eta)$.
 \end{lem}
   
We construct a \emph{minimizing path}, that is a sequence $(A_\eta)_{\eta \in \R} \subseteq \A$ with $A_\eta \in M(\eta)$ for all $\eta \in \R$, and will define the conditional functional of $Y$ given $\A$ as the inverse of the path. Firstly, we show that we can choose our path to be decreasing using Zorn's Lemma. 

\begin{prop}\label{prop:minimizing_path_is_decreasing}
There exists a decreasing minimizing path, i.e.\ there exists a sequence $(A_\eta)_{\eta \in \R}$ with $A_\eta \in M(\eta)$ for any $\eta \in \R$ such that $A_{\eta'} \subseteq A_\eta$ for any $\eta < \eta'$.
\end{prop}
For functionals with strictly increasing identification functions, any minimizing path will automatically be almost surely decreasing as shown in the next section.  

\begin{defn}\label{def:conditional_functional}
Let $(A_\eta)_{\eta \in \R} \subseteq \A$ be a decreasing minimizing path. We define the \emph{conditional functional $T$ of $Y$ given $\A$} as
\begin{equation}\label{eq:def_conditional_functional}
    T(Y \mid \A)(\omega) =\sup\{\eta \in \R \mid \omega \in A_\eta\}, \quad \omega \in \Omega.
\end{equation}
\end{defn}

\begin{rem}\label{rem:implication_of_decreasing_path}
Since we consider a decreasing minimizing path $(A_\eta)_{\eta \in \R}$, we may write $T(Y \mid \A)$ equivalently by the inf-representation
\begin{equation*}
    T(Y \mid \A)(\omega) = \inf\{\eta \in \R \mid \omega \notin A_\eta\}, \quad \omega \in \Omega.
\end{equation*}
\end{rem}
\begin{rem}\label{rem:conditional_functional_may_be_defined_over_dense_Q}
    We note that 
    \begin{equation}\label{eq:def_conditional_functional_with_Q}
    T(Y \mid \A)(\omega) =\sup\{\eta \in \Q \mid \omega \in A_\eta\}
\end{equation}
for all $\omega \in \Omega$ since $\Q$ lies dense in $\R$. Using equation \eqref{eq:def_conditional_functional_with_Q} instead of \eqref{eq:def_conditional_functional} will be useful for treating null sets and show measurability conditions. 
\end{rem}

Some words of caution are necessary here. We defined $T(Y\mid \A)$ as the inverse of a particular minimizing path and hence $T(Y\mid \A)$ will be path dependent in general. This is no surprise if we consider that e.g.\ the quantile of a distribution is in general not unique either. In the next section, we will show that the conditional functional is almost surely unique if the corresponding identification function is strictly increasing in its first argument as it is the case for the mean functional. In Section \ref{Subsec:cond_quantiles}, we study conditional quantiles and show that we may obtain the lower quantile if we choose the minimizing path accordingly. In a first step, it is important to note that $T(Y\mid \A)$ as defined in \eqref{eq:def_conditional_functional} is actually a solution to our problem.

\begin{prop}\label{prop:conditional_functional_is_valid}
    The conditional functional $T(Y \mid \A)$ is $\A$-measurable and minimizes \eqref{eq:Vmin} over all $\A$-measurable random variables $X$ for all $\eta \in \R$. If the elementary loss functions $S_\eta$ defined at \eqref{eq:Seta} are nonnegative, then the expectation in \eqref{eq:minimizing_criterion_conditional_functionals} always exists, and $T(Y \mid \A)$ minimizes \eqref{eq:minimizing_criterion_conditional_functionals} over all $\A$-measurable random variables $X$.
\end{prop}

\begin{rem}\label{rem:conditional_functionals_for_sigma_algebras}
    Since $\s$-algebras are also $\s$-lattices, the proposed construction of conditional functionals  also applies if $\A\subseteq \F$ is a $\s$-algebra. In this case, Proposition \ref{prop:conditional_functional_is_valid} shows that our construction provides an alternative characterization of the usual notion of conditional functionals defined as functionals of the conditional distribution.
    %The only difference is that we minimize $s_A(\eta)$ over a possibly larger family of sets in this case .\\   
\end{rem}

\begin{rem}
The condition of nonnegative elementary loss funcitons is satisfied for all interesting identification functions we can think of, so in particular for the mean and for quantiles. 
\end{rem}

Finally assume that the $\s$-lattice, respectively the $\s$-algebra, $\A$ is generated by some random element $\xi$ that lies in an ordered metric space $(\XX, d,\preceq)$, that is $\A=\LL(\xi)$ or $\A=\s(\xi)$. By Proposition \ref{prop:factorization_for_random_elements} and the classical factorization result, we know that there exists an increasing Borel measurable function $f:\XX \to \R$ and a Borel measurable function $g:\XX \to \R$ such that 
\begin{equation}\label{eq:Optimal_solution_representation}
        \E S(h(\xi), Y)
    \end{equation}
is minimized by $f$ over the class of all Borel measurable and increasing functions $h:\XX \to \R$, and minimized by $g$ over all Borel measurable functions $h:\XX \to \R$. Representation \eqref{eq:Optimal_solution_representation} of the underlying minimizing problem shows the connection to the work of \cite{Optimal} more clearly, where the authors minimize the empirical version of \eqref{eq:Optimal_solution_representation} over all increasing functions $h$.

\subsection{Strictly increasing identification functions and conditional expectations}\label{Subsec:cond_expectations}
Throughout this section, we assume that the identification function $V$ is strictly increasing in its first component. Under this assumption, the conditional functional is almost surely unique and we use this fact to show the part of Theorem \ref{thm:universality_of_ICL} concerning the mean functional. Let $(A_\eta)_{\eta \in \R}$ be any minimizing path with $A_\eta \in M(\eta)$ for all $\eta \in \R$. 

\begin{lem}\label{lem:A_eta_is_automatically_decreasing_for_strictly_increasing_identification_functions} The sequence $(A_\eta)_{\eta \in \R}$ is almost surely decreasing, i.e.\ $\prob(A_{\eta'} \setminus A_{\eta})=0$ for $\eta < \eta'$.
\end{lem}
%With Lemma \ref{lem:A_eta_is_automatically_decreasing_for_strictly_increasing_identification_functions} we may show uniqueness of the conditional functional. 
\begin{prop}\label{prop:conditional_functional_does_not_depend_on_path_if_the_identification_function_is_strictly_increasing}
    The conditional functional $T(Y \mid \A)$ as defined in \eqref{eq:def_conditional_functional} is almost surely unique. That is, for any minimizing paths $(A_\eta)_{\eta \in \R},(\tilde{A}_\eta)_{\eta \in \R} \subseteq \A$ with corresponding conditional functionals $X$ and $\tilde{X}$ defined at \eqref{eq:def_conditional_functional}, it holds that $X=\tilde{X}$ almost surely.
\end{prop}

For the remainder of this section we focus on the mean functional with identification function $V(\eta,y)=\eta-y$. By Lemma \ref{lem:A_eta_is_automatically_decreasing_for_strictly_increasing_identification_functions}, any minimizing path $(A_\eta)_{\eta \in \R}$ is almost surely decreasing and by Proposition \ref{prop:conditional_functional_does_not_depend_on_path_if_the_identification_function_is_strictly_increasing}, the conditional expectation $\E(Y \mid \A)$ defined by \eqref{eq:def_conditional_functional} is almost surely unique. For any square integrable random variable $Y$, \cite{Brunk_1965} defined the conditional expectation given $\A$ as its projection onto the closed convex cone of square integrable and $\A$-measurable random variables, that is as the minimizer of $\E(X-Y)^2$ over all $\A$-measurable random variables $X \in L_2$. Moreover, \citet[Theorem 3.2.]{Brunk_1965} shows that $\E(Y \mid \A)$ minimizes $\E S(X,Y)$ over all over all $\A$-measurable random variables $X \in L_2$ for any Bregman scoring function $S$, where the class of Bregman scoring functions essentially comprises all consistent scoring functions for the mean functional, see \cite{Optimal}. These considerations show that Definition \ref{defn:cond_expectations} coincides with Definition \ref{def:conditional_functional} for the mean functional for square integrable random variables. For $X \in L_1 \setminus L_2$, we may approximate $X$ by $(X_n)_{n \in \N}\subseteq L_2$ as it is done in the proof of Theorem 2 by \cite{Brunk_1963}, to show that the definitions also correspond for integrable random variables whenever the given expectations exist. 

We conclude this section by showing the first part of Theorem \ref{thm:universality_of_ICL}.

\begin{proof}[Proof of the first part of Theorem \ref{thm:universality_of_ICL}] %For any integrable random variable $Y$ and $\s$-lattice $\A \subseteq \F$, it follows by uniqueness of the conditional expectations that $\E(Y\mid \A)$ minimizes $\E S_\eta(G,Y)$ over all $\A$-measurable random variables $X$ for any $\eta \in \R$, where $S_\eta$ denotes the elementary scoring function for the mean functional. 
For $\eta \in \R$, consider the elementary scoring function for the mean $S_\eta^E$ as defined in \eqref{eq:elementary_loss_function_for_the_mean}. By the above considerations, it follows directly that $P_{Y \mid \A}(\cdot, (z,\infty))=\prob(Y>z \mid \A)$ minimizes $\E S_\eta(X, \one_{\{Y>z\}})$ over all ${\A}$-measurable random variables $X$ for any $z \in \R$ and thus the second claim of the first part follows by Fubini. Note that analogously one can show that for any $z\in \R$, $P_{Y \mid \A}(z)$ minimizes $\E S(X,\one_{\{Y\leq z\}})$ over all $\overline{\A}$-measurable random variables $X$. Consider now any $\s$-finite measure $H$ on $\R \times \R$. By the disintegration theorem \citep[Theorem 1.23]{Kallenberg2017}, there exists a $\s$-finite measure $P$ on $\R$ and a kernel $K$ from $(\R, \BR)$ to $(\R, \BR)$ such that $H=P \otimes K$. Hence, we may derive by Fubini
\begin{align*}
    \E \myS (F,Y)&= \E \int_{\R \times \R} S_{\eta}^E(F(z),\one_{\{Y\le z\}}) \diff H(z,\eta)\\ &= \E \int_\R \int_{\R}S_{\eta}^E(F(z),\one_{\{Y\le z\}}) K(z,\mathrm{d} \eta) \diff P(\eta)\\ &=\int_\R \int_{\R} \E S_{\eta}^E(F(z),\one_{\{Y\le z\}}) K(z,\mathrm{d} \eta) \diff P(\eta),
\end{align*}
which concludes the proof of the first part.
\end{proof}

\subsection{Conditional quantiles and quantile calibration of ICL}\label{Subsec:cond_quantiles}
Our main motivation to study conditional functionals was to show that ICL is quantile calibrated (Proposition \ref{prop:ICL_is_quantile_calibrated}) and to show universality of ICL (Theorem \ref{thm:universality_of_ICL}). Now, we turn our attention to quantiles and modify our notation slightly: For $\alpha \in (0,1)$ and $\eta \in \R$, denote by $M(\alpha, \eta)$ the family of all $A\in \A$ that minimize $\E (\one_A V^\alpha(\eta,Y))$, where $V^\alpha$ denotes the identification function for the $\alpha$-quantile as given at \eqref{eq:identification_function_quantile}. Firstly, we show that the minimizing paths may be chosen such that they additionally order with respect to $\alpha$. 
 
\begin{lem} \label{lem:existence_of_valid_rectangular_array}There exists a rectangular path 
    $(A_\eta^\alpha)^{\alpha \in (0,1)}_{\eta \in \R} \subseteq \A$ which is decreasing in $\eta$ and increasing in $\alpha$, that is for any $\alpha \in (0,1)$ and $\eta \in \R$, we have $A_\eta^\alpha \in M(\alpha, \eta)$, the sequence $(A_\eta^\alpha)_{\alpha \in (0,1)}$ is increasing and the sequence $(A_\eta^\alpha)_{\eta \in \R}$ is decreasing.
\end{lem}
In the proof of Lemma \ref{lem:existence_of_valid_rectangular_array}, we first construct a rectangular array $(\hat{A}^{\alpha}_\eta)$ for $\eta \in \R$ and $\alpha \in (0,1)\cap \Q$ with the given monotonicity constraints and extend it in a second step to a rectangular array 
\begin{equation*}\label{eq:particular_extension_of_rectangular_arry_text}
        A_\eta^\alpha = \bigcup_{\substack{\beta \in (0,1)\cap \Q: \beta < \alpha }} \hat{A}_\eta^\beta, \quad \alpha \in (0,1), \eta \in \R. 
    \end{equation*}

In the next lemma, we show that this particular extension implies that the resulting $\alpha$-quantiles are left-continuous. If we were interested in upper-quantiles rather than lower-quantiles we could extend $(\hat{A}^{\alpha}_\eta)$ by $A_\eta^\alpha = \bigcap_{\beta \in (\alpha,1)\cap \Q} \hat{A}_\eta^\beta$, to obtain right-continuous quantiles. 
\begin{lem}\label{lem:cond_quantiles_define_valid_quantile_function}
Consider a rectangular path $(A_\eta^\alpha)^{\alpha \in (0,1)}_{\eta \in \R} \subseteq \A$ with the monotonicity constraints as given in Lemma \ref{lem:existence_of_valid_rectangular_array} which is left-continuous, that is for any $\alpha \in (0,1)$ 
\begin{equation*}
 A_\eta^\alpha = \bigcup_{\beta \in (0,1): \beta < \alpha} A_\eta^\beta.
\end{equation*}
The corresponding conditional quantiles defined by \eqref{eq:def_conditional_functional}, then define a random element $F\in \GG_\A$ given by
\begin{equation*}
F^{-1}(\alpha)= q_\alpha(Y\mid \A), \quad \alpha \in (0,1).
\end{equation*}
\end{lem}
For any $\alpha \in (0,1)$, we denote by $q_\alpha(Y \mid X)$ the lower $\alpha$-quantile of the conditional law $P_{Y \mid X}$, see e.g.\ \cite{Tomkins}. Equivalently, $q_\alpha(Y \mid X)$ results as a minimizer of $\E \QS_\alpha(Z,Y)$ over all $\s(X)$-measurable random variables $Z$ as shown by \cite{Armerin}, that is, $q_\alpha(Y \mid X)$ may equivalently be constructed as proposed in this section. To show quantile calibration of ICL, we additionally need the following lemma, which is the analogue of Lemma \ref{lem:equivalence_auto_and_iso_calibration_for_random_variables} for conditional quantiles instead of conditional expectations.

\begin{assump}\label{assump:joint_distribution}
    The joint distribution of the random vector $(X,Y)$ is absolutely continuous, or, the marginal distribution of $X$ has countable support without accumulation points, i.e.\ the support of $X$ consists of countably many isolated points. 
\end{assump}
\begin{lem}\label{lem:equivalence_for_conditional_quantiles_given_generated_sigma_lattice}
    For any random variables $X$ and $Y$ that satisfy Assumption \ref{assump:joint_distribution}, we have $q_\alpha(Y\mid X)=X$ if and only if $q_\alpha(Y\mid \LL(X))=X$.
\end{lem}
We are strongly convinced that Lemma \ref{lem:equivalence_auto_and_iso_calibration_for_random_variables} holds in full generality and not only under Assumption \ref{assump:joint_distribution}, and we conjecture that the claim may be extended to any identifiable functional. However, currently we are only able to show the claim in the stated special cases. With Lemma \ref{lem:cond_quantiles_define_valid_quantile_function} and Lemma \ref{lem:equivalence_auto_and_iso_calibration_for_random_variables} we are in position to show quantile calibration of ICL (Proposition \ref{prop:ICL_is_quantile_calibrated}) and the remaining part of Theorem \ref{thm:universality_of_ICL}.

\begin{proof}[Proof of Proposition \ref{prop:ICL_is_quantile_calibrated}]
We apply Fubini to the quantile score based representation of the CRPS 
	\begin{align*}
		\E\crps (G,Y) = \E \int_{0}^{1} \QS_\alpha (G^{-1}(\alpha),Y) \diff \alpha 
		= \int_{0}^{1} \E \QS_\alpha (G^{-1}(\alpha),Y)  \diff \alpha.
	\end{align*}
 By Proposition \ref{prop:conditional_functional_is_valid}, we know that the integrand can be minimized pointwise for each $\alpha \in (0,1)$ and by Lemma \ref{lem:cond_quantiles_define_valid_quantile_function}, it follows that the solutions may be chosen such that they define a random distribution $F\in \GG_\A$. By the uniqueness claim in Proposition \ref{prop:crps_characterization_of_ICL}, it follows that $P_{Y \mid \A}(\cdot, B)= F(\cdot, B)$ almost surely for all $B \in \BR$, and hence, for any $\alpha \in (0,1)$,
 \begin{align*}
     P_{Y \mid \A}^{-1}(\alpha)= F^{-1}(\alpha) = q_\alpha(Y \mid \A).
 \end{align*}
Using that $P_{Y \mid \A}^{-1}(\alpha)$ is $\A$-measurable, we see that
    \begin{equation*}
        P_{Y\mid \A}^{-1}(\alpha) = q_\alpha\big(Y \mid \LL(P_{Y\mid \A}^{-1}(\alpha))\big), \eqforall \alpha \in (0,1),
    \end{equation*}
since $P_{Y \mid \A}^{-1}(\alpha)$ minimizes $\E (\QS_\alpha (X,Y))$ over all $\A$-measurable random variables $X$, and hence, also over all $\LL(P_{Y \mid \A}^{-1}(\alpha))$-measurable random variables. By Lemma \ref{lem:equivalence_for_conditional_quantiles_given_generated_sigma_lattice}, it follows that \begin{equation*}
    P_{Y\mid \A}^{-1}(\alpha) = q_\alpha(Y \mid P_{Y\mid \A}^{-1}(\alpha)), \eqforall \alpha \in (0,1),
\end{equation*}
that is, $P_{Y \mid \A}$ is quantile calibrated. 
\end{proof}

\begin{proof}[Proof of the second part of Theorem \ref{thm:universality_of_ICL}]
In the proof of Proposition \ref{prop:ICL_is_quantile_calibrated}, it was shown that for any $\alpha \in (0,1)$, $P_{Y \mid \A}^{-1}(\alpha)$ is a version of $q_\alpha(Y \mid \A)$ and hence the second claim of the second part of Theorem \ref{thm:universality_of_ICL} follows directly by Fubini. Consider now any $\s$-finite measure $H$ on $(0,1) \times \R$. We may argue analogously as in Section \ref{Subsec:cond_expectations} by the disintegration theorem \citep[Theorem 1.23]{Kallenberg2017}, that there exists a $\s$-finite measure $P$ on $(0,1)$ and a kernel $K$ from $((0,1),\mathcal{B}(0,1))$ to $(\R, \BR)$ such that $M=P \otimes K$ and thus also the first claim of the second part follows by Fubini. 
\end{proof}

\begin{rem}
For any random variable $Y$ and $\s$-algebra $\A \subseteq \F$, \cite{Tomkins} calls a random variable $X$ a conditional median of $Y$ with respect to $\A$, if $X$ is $\A$-measurable and $\prob(Y>X \mid \A) \leq 1/2\leq \prob(Y \geq X \mid \A)$ almost surely. He shows existence of conditional medians by taking the lower-version of the median of $P_{Y \mid \A}$, the regular conditional distribution of $Y$ given $\A$. Clearly, this construction is not limited to the median and may be extended to any $\alpha$-quantile. However, the proof uses linearity of conditional expectations and is hence not applicable to ICL, since conditional expectations given $\s$-lattices are in general not linear. For this reason, we proposed an alternative approach for conditional quantiles given $\s$-lattices based on identification functions and scoring functions, where we have shown in the proof of Proposition \ref{prop:ICL_is_quantile_calibrated} that also for $\s$-lattices, the conditional quantiles coincide with the quantiles of ICL. The two approaches coincide in the case of $\A$ being a $\s$-algebra as it is shown e.g.\ in \cite{Armerin}. Whether this is also true in the case of $\s$-lattices is currently an open question. 
\end{rem}

\section*{Acknowledgements} 
The authors thank Lutz D\"umbgen, Tilmann Gneiting, Alexander Henzi, Alexander Jordan, Ilya Molchanov, and Eva-Maria Walz for valuable discussions and inputs. Financial support from the Swiss National Science Foundation is gratefully acknowledged.
  
\bibliographystyle{plainnat}
\bibliography{biblio}

\begin{thebibliography}{43}
\providecommand{\natexlab}[1]{#1}
\providecommand{\url}[1]{\texttt{#1}}
\expandafter\ifx\csname urlstyle\endcsname\relax
  \providecommand{\doi}[1]{doi: #1}\else
  \providecommand{\doi}{doi: \begingroup \urlstyle{rm}\Url}\fi

\bibitem[Armerin(2014)]{Armerin}
F.~Armerin.
\newblock The conditional quantile as a minimizer.
\newblock \emph{Working Paper}, 2014.

\bibitem[Arnold(2020)]{Masterthesis}
S.~Arnold.
\newblock {Isotonic distributional approximation}.
\newblock Master's thesis, Universit{\"a}t Bern, 2020.

\bibitem[Arnold et~al.(2023)Arnold, Henzi, and Ziegel]{Arnold_Henzi_Ziegel}
Sebastian Arnold, Alexander Henzi, and Johanna~F. Ziegel.
\newblock {Sequentially valid tests for forecast calibration}.
\newblock \emph{Ann. Appl. Stat.}, 17\penalty0 (3):\penalty0 1909 -- 1935,
  2023.

\bibitem[Ayer et~al.(1955)Ayer, Brunk, Ewing, Reid, and Silvermann]{Ayer1955}
Miriam Ayer, H.~D. Brunk, G.~M. Ewing, W.~T. Reid, and Edward Silvermann.
\newblock An empirical distribution function for sampling with incomplete
  information.
\newblock \emph{Ann. Math. Stat.}, 26:\penalty0 641--647, 1955.

\bibitem[Barlow et~al.(1972)Barlow, Brunk, Bartholomew, and Bremner]{4B}
R.~E. Barlow, H.~D. Brunk, D.~J. Bartholomew, and J.~M. Bremner.
\newblock \emph{Statistical {I}nference under {O}rder {R}estrictions}.
\newblock Wiley, 1972.

\bibitem[Bauer(1996)]{Bauer_Prob}
H.~Bauer.
\newblock \emph{Probability Theory}.
\newblock De Gruyter, Berlin, New York, 1996.

\bibitem[Bauer(2001)]{Bauer_MI}
H.~Bauer.
\newblock \emph{Measure and Integration Theory}.
\newblock De Gruyter, Berlin, New York, 2001.

\bibitem[Bogachev(2018)]{Bogachev2018}
V.~I. Bogachev.
\newblock \emph{Weak Convergence of Measures}.
\newblock Am. Math. Soc., 2018.

\bibitem[Brunk(1961)]{Brunk_1961}
H.~D. Brunk.
\newblock Best fit to a random variable by a random variable measurable with
  respect to a $\sigma $-lattice.
\newblock \emph{Pacific J. Math.}, 11\penalty0 (3):\penalty0 785--802, 1961.

\bibitem[Brunk(1963)]{Brunk_1963}
H.~D. Brunk.
\newblock On an extension of the concept conditional expectation.
\newblock \emph{Proc. Am. Math. Soc.}, 14\penalty0 (2):\penalty0 298--304,
  1963.

\bibitem[Brunk(1965)]{Brunk_1965}
H.~D. Brunk.
\newblock Conditional expectation given a $\s$-lattice and applications.
\newblock \emph{Ann. Math. Stat.}, 36\penalty0 (5):\penalty0 1339--1350, 1965.

\bibitem[Brunk and Johansen(1970)]{Brunk1970}
H.~D. Brunk and S.~Johansen.
\newblock {A generalized Radon-Nikodym derivative.}
\newblock \emph{Pacific J. Math.}, 34\penalty0 (3):\penalty0 585 -- 617, 1970.

\bibitem[Dawid(1984)]{Dawid1984}
A.~P. Dawid.
\newblock Statistical theory: The prequential approach.
\newblock \emph{J. R. Stat. Soc., A: Stat. Soc}, 147:\penalty0 278--290, 1984.

\bibitem[Diebold et~al.(1998)Diebold, Gunther, and Tay]{Diebold1998}
F.~X. Diebold, T.~A. Gunther, and A.~S. Tay.
\newblock Evaluating density forecasts with applications to financial risk
  management.
\newblock \emph{Int. Econ. Rev.}, 39:\penalty0 863--883, 1998.

\bibitem[Ehm et~al.(2016)Ehm, Gneiting, Jordan, and Kr{\"u}ger]{Ehm_et_al}
W.~Ehm, T.~Gneiting, A.~Jordan, and F.~Kr{\"u}ger.
\newblock Of quantiles and expectiles: Consistent scoring functions, choquet
  representations and forecast rankings.
\newblock \emph{J. R. Stat. Soc., B: Stat. Methodol.}, 78\penalty0
  (3):\penalty0 505--562, 2016.

\bibitem[{El Barmi} and Mukerjee(2005)]{El-BarmiMukerjee2005}
H.~{El Barmi} and H.~Mukerjee.
\newblock Inferences under a stochastic ordering constraint.
\newblock \emph{JASA}, 100:\penalty0 252--261, 2005.

\bibitem[Embrechts and Hofert(2013)]{Generalized_inverses}
P.~Embrechts and M.~Hofert.
\newblock A note on generalized inverses.
\newblock \emph{Math. Methods Oper. Res.}, 77\penalty0 (3):\penalty0 423--432,
  2013.

\bibitem[Fissler and Holzmann(2022)]{Fissler_2022}
T.~Fissler and H.~Holzmann.
\newblock Measurability of functionals and of ideal point forecasts.
\newblock \emph{Electron. J. Stat.}, 16\penalty0 (2), 2022.

\bibitem[Gneiting(2011)]{Gneiting_2011}
T.~Gneiting.
\newblock Making and evaluating point forecasts.
\newblock \emph{JASA}, 106\penalty0 (494):\penalty0 746--762, 2011.

\bibitem[Gneiting and Katzfuss(2014)]{Gneiting_Katzfuss_2014}
T.~Gneiting and M.~Katzfuss.
\newblock Probabilistic forecasting.
\newblock \emph{Annu. Rev. Stat. Appl.}, 1:\penalty0 125--151, 2014.

\bibitem[Gneiting and Raftery(2007)]{Gneiting_Raftery_2007}
T.~Gneiting and A.~E. Raftery.
\newblock Strictly proper scoring rules, prediction, and estimation.
\newblock \emph{JASA}, 102\penalty0 (477):\penalty0 359--378, 2007.

\bibitem[Gneiting and Ranjan(2013)]{GneitingRanjan2013}
T.~Gneiting and R.~Ranjan.
\newblock Combining predictive distributions.
\newblock \emph{Electron. J. Stat.}, 7:\penalty0 1747--1782, 2013.

\bibitem[Gneiting and Resin(2022)]{Tilmann_Johannes_Calibration}
T.~Gneiting and J.~Resin.
\newblock Regression diagnostics meets forecast evaluation: Conditional
  calibration, reliability diagrams, and coefficient of determination.
\newblock \emph{Preprint, \url{arXiv:2108.03210}}, 2022.

\bibitem[Guo et~al.(2017)Guo, Pleiss, Sun, and Weinberger]{Guo_et_al_2017}
C.~Guo, G.~Pleiss, Y.~Sun, and K.~Q. Weinberger.
\newblock On calibration of modern neural networks.
\newblock \emph{Preprint, \url{arXiv:1706.04599}}, 2017.

\bibitem[Gupta et~al.(2022)Gupta, Podkopaev, and Ramdas]{Gupta_et_al_2022}
C.~Gupta, A.~Podkopaev, and A.~Ramdas.
\newblock Distribution-free binary classification: Prediction sets, confidence
  intervals and calibration.
\newblock \emph{Preprint, \url{arXiv:2006.10564}}, 2022.

\bibitem[Henzi and Ziegel(2021)]{Alexander_Johanna_Biometrika}
A.~Henzi and J.F. Ziegel.
\newblock {Valid sequential inference on probability forecast performance}.
\newblock \emph{Biometrika}, 109\penalty0 (3):\penalty0 647--663, 09 2021.

\bibitem[Henzi et~al.(2021)Henzi, Ziegel, and Gneiting]{IDR}
A.~Henzi, J.F. Ziegel, and T.~Gneiting.
\newblock Isotonic distributional regression.
\newblock \emph{J. R. Stat. Soc., B: Stat. Methodol.}, 83\penalty0
  (5):\penalty0 963--993, 2021.

\bibitem[Henzi et~al.(2023)Henzi, Kleger, and Ziegel]{DIM}
A.~Henzi, G.~Kleger, and J.~F. Ziegel.
\newblock Distributional (single) index models.
\newblock \emph{JASA}, 118\penalty0 (541):\penalty0 489--503, 2023.

\bibitem[Hong(2014)]{Hong}
L.~Hong.
\newblock The linear topology associated with weak convergence of probability
  measures.
\newblock \emph{Preprint, \url{arXiv:1312.6589}}, 2014.

\bibitem[Jordan et~al.(2022)Jordan, M{\"u}hlemann, and Ziegel]{Optimal}
A.I. Jordan, A.~M{\"u}hlemann, and J.F. Ziegel.
\newblock Characterizing the optimal solutions to the isotonic regression
  problem for identifiable functionals.
\newblock \emph{Ann. Inst. Stat. Math.}, 74:\penalty0 489--514, 2022.

\bibitem[Kallenberg(2017)]{Kallenberg2017}
O.~Kallenberg.
\newblock \emph{Random Measures, Theory and Applications}.
\newblock Springer, Switzerland, 2017.

\bibitem[Kuenzi(1969)]{Kuenzi}
N.~J. Kuenzi.
\newblock \emph{An investigation of some problems concerning the additivity of
  conditional expectation with respect to a sigma-lattice}.
\newblock PhD thesis, University of Iowa, 1969.

\bibitem[Lee(1975)]{Chu-in}
C.I.C. Lee.
\newblock \emph{{Conditional Expectation given a $\sigma$-lattice, Isotonic
  Regression, Univariate and Multivariate}}.
\newblock PhD thesis, Oregon State University, 1975.

\bibitem[Matheson and Winkler(1976)]{Matheson1976}
J.~E. Matheson and R.~L. Winkler.
\newblock Scoring rules for continuous probability distributions.
\newblock \emph{Manag. Sci.}, 22:\penalty0 1087--1096, 1976.

\bibitem[M{\"o}sching and D{\"u}mbgen(2020)]{Mosching2022}
A.~M{\"o}sching and L.~D{\"u}mbgen.
\newblock {Monotone least squares and isotonic quantiles}.
\newblock \emph{Electron. J. Stat.}, 14\penalty0 (1):\penalty0 24 -- 49, 2020.

\bibitem[Mosler and Mozharovskyi(2022)]{Statistical_depth}
Karl Mosler and Pavlo Mozharovskyi.
\newblock Choosing among notions of multivariate depth statistics.
\newblock \emph{Stat. Sci.}, 37\penalty0 (3):\penalty0 348--368, 2022.

\bibitem[Richmond(2020)]{Richmond}
T.~Richmond.
\newblock \emph{General Topology: An Introduction}.
\newblock De Gruyter, Berlin, New York, 2020.

\bibitem[Schulz and Lerch(2022)]{Schulz_Lerch_2022}
B.~Schulz and S.~Lerch.
\newblock Machine learning methods for postprocessing ensemble forecasts of
  wind gusts: A systematic comparison.
\newblock \emph{Mon. Weather Rev.}, 150\penalty0 (1):\penalty0 235 -- 257,
  2022.

\bibitem[Str{\"a}hl and Ziegel(2017)]{Strahl_Ziegel_2017}
C.~Str{\"a}hl and J.F. Ziegel.
\newblock {Cross-calibration of probabilistic forecasts}.
\newblock \emph{Electron. J. Stat.}, 11\penalty0 (1):\penalty0 608 -- 639,
  2017.

\bibitem[Thorarinsdottir et~al.(2016)Thorarinsdottir, Scheuerer, and
  Heinz]{Thorarinsdottir_et_al_2016}
T.~L. Thorarinsdottir, M.~Scheuerer, and C.~Heinz.
\newblock Assessing the calibration of high-dimensional ensemble forecasts
  using rank histograms.
\newblock \emph{J. Comput. Graph. Stat.}, 25:\penalty0 105--122, 2016.

\bibitem[Tomkins(1975)]{Tomkins}
R.~J. Tomkins.
\newblock {On conditional medians}.
\newblock \emph{Ann. Probab.}, 3\penalty0 (2):\penalty0 375 -- 379, 1975.

\bibitem[Tsyplakov(2013)]{Tsyplakov2013}
A.~Tsyplakov.
\newblock Evaluation of probabilistic forecasts: {P}roper scoring rules and
  moments.
\newblock \emph{Preprint}, 2013.

\bibitem[Ziegel and Gneiting(2014)]{Ziegel_Gneiting_2014}
J.F. Ziegel and T.~Gneiting.
\newblock Copula calibration.
\newblock \emph{Electron. J. Stat.}, 8:\penalty0 2619--2638, 2014.

\end{thebibliography}
\appendix

\section{Proof of Theorem \ref{thm:existence_and_uniqueness_ICL}}\label{appendix:proof_of_existence_and_uniqueness_of_ICL}

\begin{proof}[Sketch of proof for Theorem \ref{thm:existence_and_uniqueness_ICL} ]
 Consider the family of finite unions of half-open intervals with rational endpoints	$\RR= \big\{\bigcup_{i=1}^n(a_i,b_i] \mid n \in \mathbb{N}, \hspace{0.1cm} a_i,b_i \in \mathbb{Q}, a_i < b_i \hspace{0.1cm} \textrm{for }1\leq i \leq n\big\}$, which is a ring and a countable generator of $\BR$. For any $q \in \Q$, we let  $P( (q,\infty) \mid \A)$ be a version of $\prob( Y > q \mid \A)$. %and write $P( \omega (q,\infty) \mid \A)$ for $\prob( Y > q \mid \A)(\omega)$ for $\omega \in \Omega$. 
 Using the properties of conditional expectations given $\s$-lattices, we first may construct a nullset $N_0 \in \F$ such that the mapping 
		\begin{equation*}
			\Q \rightarrow [0,1], \quad q \mapsto P( (q,\infty) \mid \A) (\omega)
		\end{equation*}
		is decreasing and nonnegative for all $\omega \in \Omega \setminus N_0$. 
		For $a,b \in \Q$ with $a<b$, we define 
		\begin{equation*}
	\tilde{P}((a,b])= P( (a,\infty) \mid \A) - P( (b,\infty) \mid \A).
		\end{equation*} 
		We then extend $\tilde{P}$ to any $A \in \RR$, where we have to show that $\tilde{P}(A)$ does not depend on the representation of $A$ (of disjoint half-open intervals), and write $P(\omega, A)$ for $P(A)(\omega)$ for $\omega \in \Omega$. Consider a further nullset $N_1$ with $\tilde{P}(\omega,\R)=1$ for all $\omega \in N_1^c$. Then $N=N_0 \cup N_1$ is a nullset such that, for all $\omega \in N^c$, it holds that $B \mapsto \tilde{P}(\omega,B)$ is a content on $\RR$ with $\tilde{P}(\omega,\R)=1$. 
  
		In a next step, we have to show that $B \mapsto \tilde{P}(\omega,B)$ is $\s$-additive for all $\omega \in N^c$, i.e.\ $\tilde{P}$ is almost surely a pre-measure on $\RR$. This can be achieved using the equivalence that a content is $\s$-additive if and only if it is continuous \cite[Theorem 3.2.]{Bauer_MI}. 
  
		By the extension theorem of Carathéodory, it follows that for $\omega \in N^c$ the pre-measure $B \mapsto \tilde{P}(\omega,B)$ may be uniquely extended to a probability measure $B \mapsto Q(\omega,B)$ on $\s(\RR)=\BR$. If $N \neq \emptyset$, we let $Q(\omega, B)= \delta_{\omega_N}(B)$ for $\omega \in N$ and some fixed $\omega_N \in N$. Then, $Q$ satisfies the first property of a Markov kernel, i.e.\ $B \mapsto Q(\omega, B)$ is a probability measure for any $\omega \in \Omega$. 
  
		For any $B \in \RR$, the map $\omega \mapsto Q(\omega,B)$ is by construction a linear combination of conditional probabilities given $\A$, and hence $\BR$-measurable. To show that $\omega \mapsto Q(\omega,B)$ is $\BR$-measurable for any $B \in \BR$, we define
		\begin{equation*}
			\mathfrak{D}=\{B \in \BR \mid \omega \mapsto Q(\omega, B) \textrm{ is $\BR$-measurable}\}
		\end{equation*}
		and show that $\mathfrak{D}$ forms Dynkin system. Since $\mathfrak{R}$ is a ring, it follows that $\s(\mathfrak{R})=\mathcal{D}(\mathfrak{R})$, where $\mathcal{D}(\mathfrak{R})$ denotes the Dynkin system generated by $\mathfrak{R}$. Since $\mathfrak{R} \subseteq \mathfrak{D}$, we may conclude
		\begin{equation*}
	\BR=\s(\mathfrak{R})=\mathcal{D}(\mathfrak{R}) \subseteq \mathfrak{D} \subseteq \BR,
		\end{equation*}
		 that is, $\mathfrak{D}=\BR$ and $Q$ is indeed a Markov kernel. Moreover, it follows by construction that $\omega \mapsto Q(\omega,(y,\infty))$ is a version of $\prob( Y >y \mid \A)$ for any $y \in \R$ and hence $Q$ satisfies all requirements of Theorem \ref{thm:existence_and_uniqueness_ICL}. 
   
   Finally, assume that $P$ and $Q$ are two Markov kernels from $(\Omega, \F)$ into $(\R, \BR)$ such that $\omega \mapsto P(\omega, (y, \infty))$ and $\omega \mapsto Q(\omega, (y, \infty))$ are versions of $\prob( Y > y \mid \A)$ for all $y \in \R$. Then it follows that $ P( \cdot , A)= Q( \cdot , A)$ almost surely for any $A \in \RR$, so for any $A \in \RR$ there exists a nullset $N_A$ such that $ P( \omega , A) = Q( \omega , A)$ for all $\omega \in N_A^c$. The countable union $N= \bigcup_{A \in \RR}N_A$ is a nullset such that $P( \omega , A)=Q
   ( \omega , A)$ for all $\omega \in N^c$ and for all $A \in \RR$. It follows by \citet[Theorem 5.4.]{Bauer_MI} that the equality holds even for all $A \in \BR$, and hence, the conditional law of $Y$ given $\A$ is uniquely defined as specified in the theorem.
	\end{proof}

\section{Ordered metric spaces}\label{appendix:ordered_metric_spaces}
Consider an ordered metric space $(\XX,d,\preceq)$. For $x \in \XX$, we define the \emph{upper closure of $x$} as $\uparrow x = \{y \in \XX \mid x \preceq y\}$ and the \emph{lower closure of $x$} as $\downarrow x = \{y \in \XX \mid y \preceq x\}$. 
%Clearly $\uparrow x$ is an upper set and $\downarrow x$ is a lower set for any $x\in \XX$. 
Following \cite{Richmond}, we call $(\XX, d, \preceq)$ \emph{$T_1$-ordered} if $\uparrow x$ and $\downarrow x$ are closed for any $x \in \XX$. Equivalently, an ordered metric space is $T_1$-ordered if and only if the following continuity assumption is satisfied: If $x_n \preceq y$ ($y \preceq x_n$) for some sequence $(x_n)_{n \in \N}\subseteq \XX$ with $x_n \to x$ as $n\to \infty$, it holds that $x \preceq y$ ($y \preceq x$). We call a set $A\subseteq \XX$ \emph{convex} if $a,c \in A$ and $a \preceq b \preceq c$ imply $b \in A$ and say that an ordered metric space $(\XX, d, \preceq)$ has a \emph{convex topology} if the induced topology $\Tau$ has a basis of convex sets. To assume that $(\XX, d, \preceq)$ has a convex topology is a weaker assumption than to say that $\XX$ is a locally convex vector space, since we do not assume $\XX$ to be a vector space. Clearly, for any $a,b \in \XX$ with $a \preceq b$ the interval $[a,b]=  \uparrow a \cap \downarrow b = \{x \in \XX \mid a \preceq x \preceq b\}$ is convex. Recall that $\XX$ is called \emph{separable} if there exists a dense countable subset $M\subseteq \XX$, i.e.\ there exists a countable $M\subseteq \XX$ with $\overline{M}=\XX$, where for any $A\subseteq \XX$, we denote by $\overline{A}$ the closure of $A$. For a sequence $(x_n)_{n\in \N}\subseteq \XX$ and $x\in \XX$, we write $x_n \uparrow x$ ($x_n \downarrow x$) if $x_n \to x$ as $n\to \infty$ and $x_n \preceq x_{n+1}$ ($x_n \succeq x_{n+1}$) for all $n\in \N$. We have to impose the following assumption on $(\XX, d, \preceq)$ in order to show Lemma \ref{lem:sigma_algebras_generated_by_generated_sigma_lattice}. 
  
 \begin{assump}\label{assump:ordered_metric_space} The ordered metric space $(\XX, d, \preceq)$ is separable, $T_1$-ordered, possesses a convex topology and for each $x\in \XX$, there exists sequences $(x_n)_{n\in \N}, (x'_n)_{n\in \N}\subseteq \XX\setminus \{x\}$ with $x_n \uparrow x$ and $x'_n \downarrow x$.
 \end{assump}

The following lemma is used to show Lemma \ref{lem:sigma_algebras_generated_by_generated_sigma_lattice}.
\begin{lem}\label{lem_appendix:generator_of_Borel_sigma_algebra_in_PR} Consider an ordered metric space $(\XX,d, \preceq)$ that satisfies Assumption \ref{assump:ordered_metric_space}.
Then, the Borel $\s$-algebra on $\XX$ is generated by the family of intervals
\begin{equation*}
\EE = \big\{[x,y] \mid x,y \in \XX, x \preceq y\},
\end{equation*}
and by the family of all measurable upper sets, that is
\begin{equation*}
\BX= \s(\EE)= \s(\BX \cap \U).
\end{equation*}
\end{lem}
\begin{proof}
For $\epsilon > 0$ and a set $B \subseteq \XX$ consider its $\epsilon$-envelope 
\begin{equation*}
B^\epsilon = \big\{x \in \XX \mid d(x,B) < \epsilon\big\} \subseteq \XX,
\end{equation*}
which is an open set which contains $B$ and where the distance of any element $x \in \XX$ to the set $B$ is defined as $d(x,B)= \underset{y \in B}{\inf} \hspace{0.1cm}d(x,y)$.
Consider $(\epsilon_n)_{n \in \N} \subseteq \R$ with $\epsilon_n \downarrow 0$. Since $\XX$ is $T_1$-ordered, we have for any $x, y \in \XX$ with $x \preceq y$
\[
[x,y] = \bigcap_{n \in \N} [x,y]^{\epsilon_n} \in \s(\OO) = \BX.
\]
Conversely, consider $U \subseteq \XX$ open and pick for any $x \in U$ a convex neighbourhood $C_x$ that is contained in $U$, which exists by Assumption \ref{assump:ordered_metric_space}. For any $x \in U$, there exist sequences $(x'_n)_{n \in \N},(\tilde{x}_n)_{n \in \N} \subseteq \XX \setminus \{x\}$ such that $x'_n \uparrow x$ and $\tilde{x}_n \downarrow x$. Hence for any $x \in U$ there exists $N_x \geq 0$ such that $x'_n \in C_x$ and $\tilde{x}_n \in C_x$ for all $n \geq N_x$. Hence, $[x_{N_x}',\tilde{x}_{N_x}]$ is a convex neighbourhood of $x$ that is contained in $\s(\EE)$ according to the considerations above. We have $U = \bigcup_{x \in U} [x'_{N_x},\tilde{x}_{N_x}]$ and by separability, we may pick a countable sequence $(x_n)_{n \in \mathbb{N}} \subseteq U$ such that $U =\bigcup_{n \in \N} [x'_{N_{x_n}},\tilde{x}_{N_{x_n}}] \in \s(\EE)$. Therefore, $\BX= \s(\EE)$. 

Clearly, $[x,y]\in \sigma(\BX \cap \U)$ for any $x, y \in \XX$ with $x \preceq y$, since it is the intersection of a closed upper set and the complement of an open upper set. Hence, $\BX=\s(\EE)\subseteq \sigma(\BX \cap \U)$. The other inclusion holds trivially. 
\end{proof}

It is easy to verify that $\R^p$ with the componentwise order and the Euclidean metric and $\PR$ with the stochastic order and the weak topology satisfy Assumption $\ref{assump:ordered_metric_space}$, where we refer to \cite{Hong} or \citet[Chapter 3]{Bogachev2018} for the fact that the weak topology on $\PR$ is convex. 

Statistical depth functions measure centrality in Euclidean space with respect to a given (empirical) distribution by assigning each point a value in $[0,1]$, see e.g.\ \cite{Statistical_depth}. 
%for a summary of various depth statistics as the Mahalanobis or the Tukey depth. 
Hence, statistical depth functions may be used to identify points in some general multivariate space as elements in some bounded interval, which clearly satisfies Assumption $\ref{assump:ordered_metric_space}$.

\section{Proofs of Section \ref{Sec:Factorization_and_connection_to_classical_conditional_laws}}
\begin{proof}[Proof of Lemma \ref{lem:sigma_algebras_generated_by_generated_sigma_lattice}] It suffices to note that for any generator $\EE \subseteq 2^\XX$ of $\BX$, we have that $\s(\xi)=\s(\EE')$ for $\EE'= \{\xi^{-1}(B) \mid B \in \EE\}$ and to recall that $\BX= \s(\BX \cap \U)$ as shown in Lemma \ref{lem_appendix:generator_of_Borel_sigma_algebra_in_PR}. 

For the second claim, let $Z$ be a $\LL(\xi)$-measurable random variable, that is $\LL(Z) \subseteq \LL(\xi)$. Then $\s(Z)= \s(\LL(Z)) \subseteq \s(\LL(\xi))=\s(\xi)$ and hence $Z$ is $\s(\xi)$-measurable as well.\end{proof}

\begin{proof}[Proof of Proposition \ref{prop:factorization_for_random_elements}] 
Sufficiency: Assume that there exists an increasing Borel measurable function $f: \XX \to \R$ such that $Z=f(\xi)$. Then, for any $a\in \R$,
\begin{align*}
\{Z>a\}=\{f(\xi)>a\}=\{\xi \in f^{-1}\big((a,\infty)\big)\} = \xi^{-1} (f^{-1}((a,\infty))) \in \LL(\xi),
\end{align*}
where we used in the last step that $f^{-1}((a,\infty)) \in \BX \cap \U$ for any $a \in \R$. \\
Necessity: We proceed by measure theoretic induction with respect to $Z$ assuming that $Z$ is $\LL(\xi)$-measurable. Let $Z=\one_L$ for $L \in \LL(\xi)$. That is, there exists $B \in \BX \cap \U$ such that $L=\xi^{-1}(B)$. Then, $Z=\one_L=\one_{\xi^{-1}(B)}= \one_{B}(\xi)$, 
and hence, we found $f= \one_{B}$ which is an increasing and Borel measurable function since $B\in \BX \cap \U$. The argument is analogous for a simple nonnegative $Z$. So let $Z \geq 0$ and $\LL(\xi)$-measurable. By \citet[Lemma 1.6]{Masterthesis}, there exists a sequence of simple nonnegative $\LL(\xi)$-measurable functions $(Z_n)_{n\in \N}$ such that $Z_n \uparrow Z$. For any $n \in \N$, there exists an increasing and Borel measurable function $f_n$ such that $Z_n= f_n(\xi)$. We define $f$ to be the pointwise limit of $(f_n)_{n \in \N}$, which is increasing and Borel measurable as well by \citet[Lemma 1.2]{Masterthesis}, and hence,
\begin{equation*}
Z= \lim\limits_{n \to \infty}Z_n=\lim\limits_{n \to \infty}f_n(\xi)=f(\xi).
\end{equation*}
To conclude let $Z$ be an arbitrary random variable that is $\LL(\xi)$-measurable. We divide $Z$ into $Z_+$ and $Z_-$. By \citet[Corollary 1.4]{Masterthesis} $Z_+$ is non-negative and $\LL(\xi)$-measurable, and $Z_-$ is nonnegative and $\overline{\LL(\xi)}$-measurable, where $\overline{\LL(\xi)}$ is the $\s$-lattice that contains all complements of elements in $\LL(\xi)$. We may proceed analogously as above to show that for a non-negative and $\overline{\LL(\xi)}$-measurable random variable $\tilde{Z}$ there exists a decreasing and Borel measurable function $g$ such that $g(\xi)=\tilde{Z}$. Hence, there exists an increasing and Borel measurable function $f_+$ such that $Z_+=f_+(\xi)$ and a decreasing and Borel measurable function $f_-$ such that $Z_-=f_-(\xi)$. The function $f=f_+-f_-$ is increasing and Borel measurable with $Z= Z_+-Z_-=f_+(\xi)-f_-(\xi)=f(\xi)$.
\end{proof}

\begin{proof}[Proof of Theorem \ref{thm:connection_to_classical_cond_laws}] Let $A \subseteq \XX$, denote the support of $P_\xi$.
For any $y\in \Q$, consider $g_y: \XX \to \R$ as given before the theorem. Then $P_{Y \mid \xi=x}((y,\infty))=g_y(x)$ for all $x \in A$. For $y \in \R\setminus \Q$, choose $(y_n)_{n \in \N}\subseteq \Q$ with $y_n \uparrow y$ and let $g_y=\lim_{n\to \infty}g_{y_n}$. Then, $P_{Y \mid \xi=x}((y,\infty))=g_y(x)$ for all $y \in \R$ and all $x \in A$. Assume that $P_{Y \mid \xi}=P_{Y \mid \LL(\xi)}$, then, for any $y\in \R$,
\begin{equation*}
\prob(Y > y \mid \xi)= \prob(Y > y \mid \LL(\xi))= g_y(\xi) \quad \textrm{almost surely,}
\end{equation*}
and hence, $g_y$ is increasing for any $y\in \R$ by Proposition \ref{prop:factorization_for_random_elements}. So for any $y\in \R$ and for any $x,x' \in A$ with $x \preceq x'$, we have $P_{Y \mid \xi= x}((y,\infty))= g_y(x)\leq g_y(x') = P_{Y \mid \xi= x'}((y,\infty))$. % and hence $P_{Y \mid \xi= x} \st P_{Y \mid \xi= x'}$.  

For the reverse direction, let $y \in \R$ be fixed and consider the $\s(\xi)$-measurable random variable $P_{Y \mid \xi}(\cdot, (y, \infty)) = g_y(\xi)$. By assumption, $g_y$ is increasing on $A$. 
%Since there exists an increasing $\tilde{g}_y: \XX \to \R$ such that $\tilde{g}_y(x)=g_y(x)$ for all $x\in A$, we may assume that $g_y$ is increasing on whole $\XX$.
% The function
%\begin{align*}
%    \tilde{g}_y(x) = \begin{cases}
%        g_y(x) & x \in A, \\
%        \sup \{g(y) \mid y \preceq x, y \in A\}, & x \notin A
%    \end{cases}
%\end{align*}
%does the job 
%No, this function does not do the job because the sup could be +\infty. However, if f is increasing on the support of \xi, then f(\xi) is L(\xi) measurable. This follows by looking at the proof of Prop 3.3. Using Prop 3.3, we can therefore conclude that increasing functions can always be extended but we don't need this here. I think it is fine to be a bit implicit here.
Hence, $P_{Y \mid \xi}(\cdot, (y, \infty))$ is $\LL(\xi)$-measurable by Proposition \ref{prop:factorization_for_random_elements}. We have $\E([\one_{\{Y >y\}}-P_{Y \mid \xi}(\cdot, (y, \infty))]\one_A)= 0$ for any $A \in \LL(\xi)$ since $\LL(\xi)\subseteq \s(\xi)$. Moreover $\E([\one_{\{Y >y\}}-P_{Y \mid \xi}(\cdot, (y, \infty))]\one_B)= 0$ for any $B\in \s(P_{Y \mid \xi}(\cdot, (y, \infty)))$ and hence $P_{Y \mid \xi}(\cdot, (y, \infty))=\prob(Y > y \mid \LL(\xi))$ for any $y \in \R$. 
\end{proof}

\section{Probabilistic calibration of ICL}\label{appendix:PIT_calibration}

Let $Y$ be a random variable and $F$ a random distribution. Suppose that the distribution $F$ is almost surely absolutely continuous and we consider it as a probabilistic forecast for $Y$. It is called \emph{probabilistically calibrated} for $Y$ if the probability integral transform $F(Y)$ is following a standard uniform distribution \citep{GneitingRanjan2013}. If the predictive distribution is not absolutely continuous, it is called \emph{probabilistically calibrated} if, with a suitable randomization at jump points of $F$, the probability integral transform is standard uniform. The following result shows that this is the case for ICL, and is due to Alexander Henzi. The interested reader may like to compare the proposition to \citet[Theorem 2.16]{Tilmann_Johannes_Calibration}.

Adapting the definition of lower-quantile calibration given at \eqref{defn_quantile_calibration}, we call $F$ \emph{upper-quantile calibrated} if \begin{equation*}
    F_u^{-1}(\alpha)= \Tilde{q}_\alpha\big(Y \mid F_u^{-1}(\alpha)\big) \quad \textrm{almost surely for all }\alpha \in (0,1),
\end{equation*}
where $F_u^{-1}$ denotes the upper-quantile function $F_u^{-1}(\alpha)=\sup\{z \in \R \mid F(z)\leq \alpha\}$ for $\alpha\in (0,1)$, and $\Tilde{q}_\alpha(Y \mid F_u^{-1}(\alpha))$ denotes the upper $\alpha$-quantile of the conditional law of $Y$ given $F_u^{-1}(\alpha)$. One might show that ICL is upper-quantile calibrated exactly analogously as in Section \ref{Subsec:cond_quantiles} by choosing the sequence of minimizing sets given before Lemma \ref{lem:cond_quantiles_define_valid_quantile_function}.

\begin{prop}
    Let $Y$ be a random variable and $F$ a random predictive distribution for $Y$, which is lower- and upper-quantile calibrated. Then,
    \begin{equation*}
        \prob\big(F(Y)< \alpha\big)\leq \alpha \leq \prob\big(F(Y-)\leq \alpha\big), \eqforall \alpha \in (0,1).
    \end{equation*}
\end{prop}
\begin{proof}
    Fix $\alpha \in (0,1)$. Since $F$ is lower-quantile calibrated and $F(Y)<\alpha$ if and only if $Y<F^{-1}(\alpha)$, we have
    \begin{align*}
        \prob\big(F(Y)< \alpha\big)&= \prob\big(Y<F^{-1}(\alpha)\big) =\prob\big(\prob\big(Y<F^{-1}(\alpha) \mid F^{-1}(\alpha) \big)\big) \leq \alpha.
    \end{align*}
Analogously by upper-quantile calibration of $F$ and since $F(Y-)\leq \alpha$ if and only if $Y \leq F_{u}^{-1}(\alpha)$,
 \begin{align*}
        \prob\big(F(Y-)\leq \alpha\big)&= \prob\big(Y \leq F_{u}^{-1}(\alpha)\big) =\prob\big(\prob\big(Y \leq F_{u}^{-1}(\alpha) \mid F_u^{-1}(\alpha) \big)\big) \geq \alpha.
    \end{align*}
\end{proof}

\section{Proofs of Sections \ref{Sec:Universality_of_ICL} and \ref{Sec:Calibration}}
 \begin{proof}[Proof of Proposition \ref{prop:crps_characterization_of_ICL}]
For any Markov kernel $G\in \GG_\A$, we may derive by Fubini that
	\begin{align*}
		\E \crps(G,Y) = \int\E (G(z)-\one_{\{Y\leq z\}})^2 \diff z 
		= \int\E \big(\one_{\{ z< Y\}}-G(\cdot, (z,\infty))\big)^2 \diff z.
	\end{align*}
The claim follows, since $P_{Y \mid \A}(\cdot, (z,\infty))$ minimizes the integrand amongst all $\A$-measurable random variables for any $z\in \R$, and uniqueness follows by Theorem \ref{thm:existence_and_uniqueness_ICL}.
\end{proof}

\begin{proof}[Proof of Proposition \ref{prop:ICL_is_threshold_calibrated}]
Fix $y \in \R$. Since $P_{Y \mid \A}(\cdot, (y, \infty))=\prob(Y > y \mid \A)$, we have
\begin{equation}\label{eq1_in_proof:ICL_is_threshold_calibrated}
		\E(P_{Y \mid \A}(\cdot, (y, \infty)) \one_B) = \E(\one_{\{Y > y\}} \one_B)
\end{equation}
for any $B \in \s(P_{Y \mid \A}(\cdot, (y, \infty)))$. Since $P_{Y \mid \A}(y) = 1 - P_{Y \mid \A}(\cdot, (y, \infty))$, equation \eqref{eq1_in_proof:ICL_is_threshold_calibrated} holds as well for any $B \in \s (P_{Y \mid \A}(y))$. It follows by elementary calculations that $\E(P_{Y \mid \A}(y) \one_B) = \E(\one_{\{Y \leq y\}} \one_B)$ for any $B \in \s (P_{Y \mid \A}(y))$ and hence $P_{Y \mid \A}(y)=\prob(Y\leq y \mid P_{Y \mid \A}(y))$, that is, the ICL $P_{Y \mid \A}$ is threshold-calibrated. 
\end{proof}

\begin{proof}[Proof of Proposition \ref{prop:implications_calibration}]
Assume firstly that $F$ is auto-calibrated and consider any $a\in \R$. We want to show that $F(\cdot, (a,\infty))$ is a version of $\prob(Y>a\mid \LL(F))$. Note that $F(\cdot,(a,\infty))$ is $\LL(F)$-measurable, since it is the composition of $F$ with $f_a:\PR \to [0,1], G \mapsto \bar{G}(a)$, where $f_a$ is Borel measurable 
%by Lemma \ref{lem:measurabilty_of_Markov_Kernels} 
and clearly increasing. By assumption $F(\cdot, (a,\infty))=\prob(Y>a \mid F)$ and hence \begin{equation}\label{eq1_proof_of_proposition_AC_implies_IC}
    \E(\one_{\{Y>a\}}\one_B)= \E(F(\cdot, (a, \infty))\one_B) \eqforall B \in \s(F).
\end{equation}
Analogously, it may be seen that $F(\cdot, (a,\infty))$ is $\s(F)$-measurable, or equivalently $\s(F(\cdot, (a,\infty)))\subseteq \s(F)$. Moreover $\LL(F)\subseteq \s(F)$ by Lemma \ref{lem:sigma_algebras_generated_by_generated_sigma_lattice}, and hence, may conclude by \eqref{eq1_proof_of_proposition_AC_implies_IC} that $F(\cdot, (a,\infty))=\prob(Y>a\mid \LL(F))$. 

Assume now that $F$ is isotonically calibrated. Then, $\prob (Y \leq y \mid P_{Y \mid \LL(F)}(y))=  P_{Y \mid \LL(F)}(y)$ for all $y\in \R$ by Proposition \ref{prop:ICL_is_threshold_calibrated}, and since the forecast $F$ is isotonically calibrated, we may conclude that $\prob (Y \leq y \mid F(y))= F(y)$ for all $y\in \R$.

Analogously, $P_{Y \mid \LL(F)}^{-1}(\alpha)=q_\alpha(Y \mid  P_{Y \mid \LL(F)}^{-1}(\alpha))$ for all $\alpha \in (0,1)$ by Proposition \ref{prop:ICL_is_quantile_calibrated}, and hence, we may conclude that $F^{-1}(\alpha)=q_\alpha(Y \mid F^{-1}(\alpha))$ for all $\alpha \in (0,1)$ since the forecast $F$ is isotonically calibrated.
\end{proof}

\begin{proof}[Proof of Lemma \ref{lem:equivalence_auto_and_iso_calibration_for_random_variables}]
 Assume that $\E(Y \mid X)= X$, that is, $\E((Y-X) \one_A) = 0$ for all $A\in\s(X)$. It follows directly by definition that $\E(Y\mid \LL(X))=X$ using the fact that $\LL(X)\subseteq \s(X)$. If $\E(Y \mid \LL(X))= X$, then $\E((Y-X) \one_A) = 0$ for all $A\in\s(X)$ by \eqref{eq:defn_cond_expectations2}, and hence, $\E(Y \mid X)=X$. 
\end{proof}

\section{Proofs of Section \ref{Sec:Conditional_Functionals}}
\begin{proof}[Proof of Lemma \ref{lem:M_eta_is_non_empty}] By definition of the set-theoretic limit inferior, 
   \begin{equation*}
       A_0= \liminf_{n\to \infty} A_n = \bigcup_{n\in \N} \bigcap_{m\geq n}A_m , 
   \end{equation*}
   is contained in $\A$, since $\A$ is a $\s$-lattice. Minimizing equation \eqref{eq:def_s_A} over all $A \in \A$ is equivalent to minimize 
    \begin{equation*}
   \E \big[\one_A V(\eta,Y)+|V(\eta,Y)|\big]
 \end{equation*}
 where the integrand is nonnegative, and hence, we can use Fatou's Lemma to show
 \begin{equation*}
      \inf_{A \in \LL}s_A(\eta) \leq s_{A_0}(\eta) \leq \liminf_{n \to \infty} s_{A_n}(\eta) = \inf_{A \in \LL} s_A(\eta).
 \end{equation*}
\end{proof} 

We need the following lemma for several of the results in Section \ref{Sec:Conditional_Functionals}.
\begin{lem} \label{lem:minimizing_path_is_decreasing}
Let $\eta, \eta' \in \R$, $\eta < \eta'$, and $A_\eta \in M(\eta)$, $A_{\eta'} \in M(\eta')$. Then, $A_\eta \cup A_{\eta'} \in M(\eta)$ and $A_{\eta} \cap A_{\eta'} \in M(\eta')$.
\end{lem}   
\begin{proof}[Proof of Lemma \ref{lem:minimizing_path_is_decreasing}]
 We have $s_{A_{\eta}}(\eta) \leq s_{A_{\eta} \cup A_{\eta'}}(\eta)$ and $s_{A_{\eta'}}(\eta') \leq s_{A_{\eta} \cap A_{\eta'}}(\eta')$. Since $A_{\eta'} \setminus (A_{\eta} \cap A_{\eta'})= A_{\eta'} \setminus A_{\eta}=(A_{\eta} \cup A_{\eta'}) \setminus A_{\eta}$, we have $v_{A_{\eta'} \setminus A_{\eta}}(\eta) \geq 0$ and $v_{A_{\eta'} \setminus A_{\eta}}(\eta') \leq 0$, which implies $v_{A_{\eta'}\setminus A_{\eta}}(\eta'')=0$ for all $\eta'' \in [\eta,\eta']$ due to the monotonicity in the first argument of the identification function. That implies $s_{A_{\eta}}(\eta)= s_{A_{\eta} \cup A_{\eta'}}(\eta)$ and $s_{A_{\eta'}}(\eta')= s_{A_{\eta} \cap A_{\eta'}}(\eta')$.
\end{proof}

\begin{proof}[Proof of Proposition \ref{prop:minimizing_path_is_decreasing}] For any $D \subseteq \R$, we call a path $A_D=(A_{D,\eta})_{\eta \in D}$ decreasing if
\begin{align*}
    A_{D,\eta'} \subseteq A_{D,\eta}, \quad \textrm{for }\eta,\eta'\in D \textrm{ with }\eta \leq \eta'.
\end{align*}
We define a partial order on all such decreasing paths on subsets of $\R$ by defining
\begin{equation*}
    A_D \preceq A'_{D'} \textrm{ if }D \subseteq D' \textrm{ and }A_{D,\eta} = A'_{D',\eta} \textrm{ for all }\eta \in D. 
\end{equation*}
Consider the partially ordered set $\Lambda$, which consists of all decreasing paths $A_D$ with $D \subseteq \R$ and $A_{D, \eta}\in M(\eta)$ for all $\eta\in D$. Given a chain in $\Lambda$, an upper bound in $\Lambda$ can be constructed by considering the function that is defined on the union of all the domains $D$. Therefore, Zorn's Lemma applies and yields that $\Lambda$ contains a maximal element $A_m$. It remains to show that the domain of $A_m$ is in fact $\R$. 

As a first step, we note that the domain $D_m$ of $A_m$ is a closed subset of $\R$, since otherwise, $A_m$ would not be maximal: For a contradiction assume that there exists $\eta_0 \in \R$ that lies on the boundary of $D_m$ but $\eta_0\notin D_m$. Assume w.l.o.g.\ that we may choose a sequence $(\eta_n)_{n\in \N}\subseteq D_m$ with $\eta_n \uparrow \eta_0$ (the case $\eta_n \downarrow \eta_0$ is analogous). We claim that
\begin{align*}
    A_{\eta_0}=\bigcap_{n\in \N} A_{m,\eta_n} \in M(\eta_0).
\end{align*}
Since $B_n \downarrow A_{\eta_0}$ as $n\to \infty$ for $B_n = \bigcap_{i=1}^n A_{m, \eta_i}$, we have that $\one_{B_n} V(\eta_n,Y)$ converges pointwise to $\one_{A_{\eta_0}}V(\eta_0,Y)$. By Lemma \ref{lem:minimizing_path_is_decreasing}, it follows that $s_{B_n}(\eta_n) \leq s_{B}(\eta_n)$ for all $B \in \A$ and all $n\in \N$ and by dominated convergence, it follows that $s_{B_n}(\eta_n) \to s_{A_{\eta_0}}(\eta_0)$ and $s_{B}(\eta_n) \to s_B(\eta_0)$ as $n\to \infty$ for all $B\in \A$. Thus we may conclude that $s_{A_{\eta_0}}(\eta_0)\leq s_{B}(\eta_0)$ for all $B \in \A$. Since $A_{\eta_0}\subseteq A_\eta $ for any $\eta<\eta_0$ and $A_{\eta}\subseteq A_{\eta_0} $ for any $\eta_0<\eta$, we could extend $A_m$ to the domain $D_m \cup \{\eta_0\}$ to obtain a greater element $\Tilde{A}_m$, which is a contradiction to $A_m$ being a maximal element. Hence, $D_m$ is closed. Suppose now that there exists a nonempty open interval $(\eta,\eta')$ contained in $D_m^c$ such that $\eta,\eta' \in D_m$. Consider $\eta_0\in (\eta,\eta')$ and any $A_{\eta_0} \in M(\eta_0)$. By Lemma \ref{lem:minimizing_path_is_decreasing}, it follows that
\begin{equation*}
    (A_{\eta_0} \cap A_{m, \eta}) \cup A_{m, \eta'} \in M(\eta_0).
\end{equation*}
So again, we could extend $A_m$ onto the domain $D_m \cup \{\eta_0\}$ to obtain a greater element $\Tilde{A}_m$, which is a contradiction to the fact that $A_m$ is a maximal element. We conclude that $D_m=\R$. 
\end{proof}

\begin{proof}[Proof of Proposition \ref{prop:conditional_functional_is_valid} ] Let $X=T(Y \mid \A)$ and $a \in \R$, then
    \begin{align*}
        \{X > a\} &= \{\omega \in \Omega \mid \sup\{\eta \in \Q: \omega \in A_\eta\}  > a\} \\
        &= \{\omega \in \Omega \mid \exists \eta \in \Q: \eta > a \textrm{ and }\omega \in A_\eta\} = \bigcup_{\eta \in \Q: \hsp \eta > a} A_\eta \in \A. 
    \end{align*}
For the second part of the first claim, it suffices by construction to show that $\prob(A_\eta \triangle \{X \geq \eta\})=0$ for all $\eta \in \R$. First note that $A_\eta \subseteq \{X \geq \eta\}$ for any $\eta \in \R$, since $\omega \in A_\eta$ implies $X(\omega)=\sup\{\eta \in \R \mid \omega \in A_\eta\} \geq \eta$. On the other hand consider $(\eta_n)_{n\in \N}\subseteq \Q$ with $\eta_n \uparrow \eta$ and $\eta_n<\eta$. Then $\prob(X>\eta_n) \downarrow \prob(X\geq \eta)$ as $n \to \infty$ by continuity of measures. Hence,
\begin{align*}
 \prob (\{X>\eta_n\} \setminus A_\eta) =   \prob \Big(\bigcup_{\substack{q \in \Q:q > \eta_n }} A_q  \setminus A_\eta\Big) 
=  \prob \Big(\bigcup_{q \in \Q:q > \eta_n} A_q  \Big) - \prob(A_\eta)\to 0, \quad\text{as }n \to \infty,
    \end{align*}
and we may conclude that $\prob(A_\eta \triangle \{X \geq \eta\})=\prob(\{X\geq\eta\} \setminus A_\eta)=0.$

If $S_\eta(x,y) \ge 0$ then the expectation at \eqref{eq:minimizing_criterion_conditional_functionals} always exists. The claim now follows by Fubini.
\end{proof}
\begin{proof}[Proof of Lemma \ref{lem:A_eta_is_automatically_decreasing_for_strictly_increasing_identification_functions}]
    By Lemma \ref{lem:minimizing_path_is_decreasing}, $s_{A_\eta}(\eta)= s_{A_\eta \cup A_{\eta'}}(\eta)$ and $s_{A_\eta'}(\eta') = s_{A_\eta \cap A_{\eta'}}(\eta')$, and hence,
\begin{equation*}
  \int_{A_{\eta'}\setminus A_\eta} V(\eta'', Y) \diff \prob =0 \eqforall\eta'' \in [\eta, \eta'].
\end{equation*}
By assumption, $V$ is strictly increasing in the first argument, so 
\begin{equation*}
  \int_{A_{\eta'}\setminus A_\eta} V(\eta', Y) -V(\eta, Y) \diff \prob =0,
\end{equation*}
and thus, $\prob(A_{\eta'} \setminus A_{\eta})=0$ since the integrand is strictly positive.

\end{proof}

For the proof of Proposition \ref{prop:conditional_functional_does_not_depend_on_path_if_the_identification_function_is_strictly_increasing}, we use the following lemma.
\begin{lem}\label{lem:auxilary_lem_to_show_uniqueness_of_cond_functional_for_strictly_increasing_identification_functions}
    For any minimizing path $(A_\eta)_{\eta \in \Q}$ and for almost all $\omega \in \Omega$ it holds that
    \begin{equation*}
 \sup\{\eta \in \Q \mid \omega \in A_\eta\}= \inf\{\eta \in \Q\mid \omega \notin A_\eta\}.
\end{equation*}
\end{lem}
\begin{proof}[Proof of Lemma \ref{lem:auxilary_lem_to_show_uniqueness_of_cond_functional_for_strictly_increasing_identification_functions}]
  For a minimizing path $(A_\eta)_{\eta \in \R}$ and $\omega \in \Omega$, let $A_\omega=\{\eta \in \Q \mid \omega \in A_\eta\}$ and $B_\omega=\{\eta \in \Q \mid \omega \notin A_\eta\}$. Clearly, $A_\omega \cap B_\omega = \emptyset$ and $A_\omega \cup B_\omega= \Q$, that is, $A_\omega$ and $B_\omega$ form a partition of $\Q$. This implies that $\sup (A_\omega)= \inf (B_\omega)$ if and only if both $A_\omega$ and $B_\omega$ are connected, that is, there exists $\eta_0 \in \Q$ such that $A_\omega=\{\eta \in \Q \mid \eta\leq \eta_0\}$ and $B_\omega=\{\eta \in \Q \mid \eta>\eta_0\}$.
   %for $x, y \in A$ (resp. in $B$) with $x \leq y$ it holds that $r \in A$ (resp. in $B$) for any rational $x \leq r \leq y$. 
   We note that $A_\omega$ (and hence as well $B_\omega$) is not connected if and only if there exists $\eta, \eta' \in \Q$ with $\eta < \eta'$ such that $\omega \in A_{\eta'}$ and $\omega \notin A_{\eta}$. Thus, we may derive that
   \begin{align*}
       \{ \omega \mid \sup\{\eta \in \Q \mid \omega \in A_\eta\}\neq \inf\{\eta \in \Q \mid \omega \notin A_\eta\}\}  
       =&\{ \omega  \mid A_\omega \textrm{ and }B_\omega \textrm{ are not connected}\}\\ 
       =& \bigcup_{\eta \in \Q} \bigcup_{\substack{\eta' \in \Q: \eta' > \eta}} (A_{\eta'} \setminus A_\eta) 
   \end{align*}
   is a nullset by Lemma \ref{lem:A_eta_is_automatically_decreasing_for_strictly_increasing_identification_functions}.
\end{proof}

\begin{proof}[Proof of Proposition \ref{prop:conditional_functional_does_not_depend_on_path_if_the_identification_function_is_strictly_increasing}]
     Consider two minimizing paths $(A_\eta)_{\eta \in \Q},(\Tilde{A}_\eta)_{\eta \in \Q} \subseteq \A$, i.e.\ $A_\eta \in M(\eta)$ and $\Tilde{A}_\eta \in M(\eta)$ for all $\eta \in \Q$, and let $X$ and $\tilde{X}$ be the conditional functionals defined at \eqref{eq:def_conditional_functional_with_Q} with respect to $(A_\eta)_{\eta \in \Q}$, $(\Tilde{A}_\eta)_{\eta \in \Q}$, respectively. By Lemma \ref{lem:auxilary_lem_to_show_uniqueness_of_cond_functional_for_strictly_increasing_identification_functions}, there are $Z, \Tilde{Z}\in \F$ with $\prob(Z^c)=\prob(\Tilde{Z}^c)=0$ and
    \begin{align*}
        X(\omega)&= \sup\{\eta \in \Q \mid \omega \in A_\eta\} = \inf\{\eta \in \Q \mid \omega \notin A_\eta\}, \quad \textrm{for all }\omega \in Z,\\
        \Tilde{X}(\omega)&= \sup\{\eta \in \Q \mid \omega \in \Tilde{A}_\eta\} = \inf\{\eta \in \Q \mid \omega \notin \Tilde{A}_\eta\}, \quad \textrm{for all }\omega \in \Tilde{Z}.
    \end{align*}
    For $\eta_0 \in \Q$, define the sequence $(\hat{A}^{\eta_0}_\eta)_{\eta \in \Q}\subseteq \A$ by
    \begin{equation}\label{eq:proof_of_uniqueness_of_conditional_functionals}
        \hat{A}^{\eta_0}_\eta = \begin{cases}
         A_\eta, & \textrm{for }\eta \leq \eta_0  \\
        \Tilde{A}_\eta, & \textrm{for }\eta > \eta_0 .
    \end{cases} 
    \end{equation}
    We can argue as in Lemma \ref{lem:A_eta_is_automatically_decreasing_for_strictly_increasing_identification_functions} to show that  $(\hat{A}^{\eta_0}_\eta)_{\eta \in \Q}$ is almost surely decreasing. Hence, it follows that
    \begin{equation*}
        N_{\eta_0}= \{ \omega \in \Omega \mid  \sup\{\eta \in \Q\mid \omega \in \hat{A}^{\eta_0}_\eta\}\neq \inf\{\eta \in \Q\mid \omega \notin \hat{A}^{\eta_0}_\eta\}\}
    \end{equation*}
    is a nullset. Let $N= \bigcup_{\eta \in \Q} N_{\eta}$ and consider $\omega \in Z \cap \Tilde{Z} \cap N^c$. Let $X(\omega)=\tau$ and $\Tilde{X}(\omega)=\Tilde{\tau}$.  Without loss of generality assume that $\tau \leq \Tilde{\tau}$ (the other case can be shown analogously with a modified sequence in \eqref{eq:proof_of_uniqueness_of_conditional_functionals}). If $\tau= \Tilde{\tau}$ there is nothing to show. Otherwise assume for a contradiction that $\tau < \Tilde{\tau}$ and choose $\tau' \in \Q$ such that $\tau < \tau' <\Tilde{\tau}$, then
    \[
        \tau = \sup\{\eta \in \Q\mid \omega \in A_\eta\} = \sup\{\eta \in \Q\mid \omega \in \hat{A}^{\tau'}_\eta\} 
        = \inf\{\eta \in \Q \mid \omega \notin \hat{A}^{\tau'}_\eta\} = \inf\{\eta \in \Q\mid \omega \notin \Tilde{A}_\eta\} = \Tilde{\tau},
    \]
    a contradiction and hence $X(\omega)=\Tilde{X}(\omega)$. Since $\omega \in Z \cap \Tilde{Z} \cap N^c$ was chosen arbitrarily and $\prob( Z \cap \Tilde{Z} \cap N^c)=1$, this concludes the proof.
\end{proof}

For the proof of Lemma \ref{lem:existence_of_valid_rectangular_array}, we need the following result.
\begin{lem}\label{lem:A_eta_order_with_alpha}
Let $0 < \alpha < \alpha' < 1$, $\eta \in \R$, $A_\eta \in M(\alpha, \eta)$ and $A'_\eta \in M(\alpha', \eta)$. Then $A_\eta \cap A'_\eta \in M(\alpha, \eta)$ and $A_\eta \cup A'_\eta \in M(\alpha',\eta)$.
\end{lem}

\begin{proof}
For $\alpha \in (0,1)$ and $\eta \in \R$, let $ s_A^\alpha(\eta) = \E(\one_A V^\alpha(\eta,Y))$ for $A\in \A$ and $v_A^\alpha(\eta) = \E(\one_A V^\alpha(\eta,Y))$ for $A\in \F$. Consider $0\leq \alpha < \alpha' \leq 1$, $\eta \in \R$, $A_\eta \in M(\alpha, \eta)$ and $A'_\eta \in M(\alpha', \eta)$. Note that $s_{A_\eta}^\alpha =  s_{A_\eta\cap A'_\eta}^\alpha + v_{A_\eta\setminus A_\eta'}^{\alpha}$ and $s_{A_\eta\cup A'_\eta}^{\alpha'} = s_{A'_\eta}^{\alpha'} +v_{A_\eta\setminus A_\eta'}^{\alpha'}$. Since $s^\alpha_{A_\eta}(\eta) \leq s^\alpha_B(\eta)$ and $s^{\alpha'}_{A'_{\eta}}(\eta) \leq s^{\alpha'}_B(\eta)$ for any $B \in \A$, it follows that
\begin{equation*}
    v_{A_\eta\setminus A_\eta'}^{\alpha} \leq 0 \leq v_{A_\eta\setminus A_\eta'}^{\alpha'}.
\end{equation*}
Since $ v_A^{\alpha'}(\eta) \leq  v_A^\alpha(\eta)$ for any $A\in \F$, we may conclude that $v_{A_\eta\setminus A_\eta'}^{\alpha}(\eta) = v_{A_\eta\setminus A_\eta'}^{\alpha'}(\eta)=0.$ 
\end{proof}

\begin{proof}[Proof of Lemma \ref{lem:existence_of_valid_rectangular_array}]
Let $(\alpha_n)_{n\in \N}$ be an enumeration of $\Q \cap (0,1)$. For each $n\in \N$, there exists a decreasing minimizing path $(\Tilde{A}^{\alpha_n}_\eta)_{\eta \in \R}$ by Lemma \ref{lem:minimizing_path_is_decreasing}. We construct inductively a rectangular array $(\hat{A}^{\alpha_n}_\eta)^{n\in \N}_{\eta \in \R}$, which is additionally monotonic in $\alpha$. W.l.o.g.\ assume that $\alpha_1 < \alpha_2$, then it follows by Lemma \ref{lem:A_eta_order_with_alpha} that the array
\begin{equation*}
    \hat{A}^{\alpha_1}_\eta = \Tilde{A}^{\alpha_1}_\eta, \quad \hat{A}^{\alpha_2}_\eta=\Tilde{A}^{\alpha_2}_\eta \cup \Tilde{A}^{\alpha_1}_\eta
\end{equation*}
is valid, decreasing in $\eta$ and increasing in $\alpha$, that is $\hat{A}^{\alpha_i}_\eta\in M(\alpha_i, \eta)$ for $i=1,2$, $\eta \in \R$ and $\hat{A}^{\alpha_1}_{\eta}\subseteq \hat{A}^{\alpha_2}_\eta$ for all $\eta \in \R$ and $\hat{A}^{\alpha_i}_{\eta'}\subseteq \hat{A}^{\alpha_i}_\eta$ for $i=1,2$ and all $\eta < \eta'.$ 

For $n\geq 3$, define $(\hat{A}^{\alpha_n}_\eta)_{\eta \in \R}$ as follows: If $\alpha_n < \gamma=\min\{\alpha_i \mid i < n\}$, let $
    \hat{A}_\eta^{\alpha_n} = \tilde{A}_\eta^{\alpha_n} \cap \hat{A}_\eta^\gamma$.
If $\alpha_n > \beta = \max\{\alpha_i \mid i < n\}$, let 
    $\hat{A}_\eta^{\alpha_n} = \tilde{A}_\eta^{\alpha_n} \cup \hat{A}_\eta^\beta$.
Otherwise, let $\beta=\max\{\alpha_i \mid i < n, \alpha_i < \alpha_n\}$, $\gamma=\min\{\alpha_i \mid i < n, \alpha_i > \alpha_n\}$ and define
\begin{equation*}
    \hat{A}^{\alpha_n}_\eta = \big(\hat{A}^{\beta}_\eta \cup \Tilde{A}^{\alpha_n}_\eta \big) \cap \hat{A}^{\gamma}_\eta.
\end{equation*}
It follows by Lemma \ref{lem:A_eta_order_with_alpha} that the resulting array $(\hat{A}^{\alpha_n}_\eta)^{n\in \N}_{\eta \in \R}$ is valid, decreasing in $\eta$ and increasing in $\alpha$. Finally, we define for $\alpha \in (0,1)$ and $\eta \in \R$,
    \begin{equation*}%\label{eq:particular_extension_of_rectangular_arry_proof}
        A_\eta^\alpha = \bigcup_{\substack{\beta \in (0,1)\cap \Q: \beta < \alpha}} \hat{A}_\eta^\beta.
    \end{equation*}
It follows directly that the array $(A_\eta^\alpha)^{\alpha \in (0,1)}_{\eta \in \R}$ is increasing in $\alpha$ and decreasing in $\eta$. So it remains to show that
    $A_\eta^\alpha \in M(\alpha, \eta)$ for all $\alpha \in (0,1)$ and all $\eta \in \R$. As in the proof of Lemma \ref{lem:A_eta_order_with_alpha}, let $ s_B^\alpha(\eta) = \E(\one_B V^\alpha(\eta,Y))$ for $B\in \A$. Fix $\alpha \in (0,1)$ and $\eta \in \R$, let $(\alpha_n)_{n \in \N}\subseteq(0,1)\cap \Q$ with $\alpha_n<\alpha$ and $\alpha_n \uparrow \alpha$ as $n\to \infty$ and
     \begin{equation*}
        A_n = \bigcup_{i=1, \dots, n} \hat{A}_\eta^{\alpha_i}.
    \end{equation*}
By Lemma \ref{lem:A_eta_order_with_alpha}, $A_n \in M(\alpha_n,\eta)$ for all $n\in \N$, and hence,  $s_{A_n}^{\alpha_n}(\eta) \leq  s_{B}^{\alpha_n}(\eta)$ for all $B \in \A$. Since $\one_{A_n} V^{\alpha_n}(\eta,Y) \to \one_{A_\eta^\alpha} V^{\alpha}(\eta,Y)$ pointwise, it follows by dominated convergence that $s^{\alpha_n}_{A_n}(\eta) \to s^{\alpha}_{A_\eta^\alpha}(\eta)$ and $s_B^{\alpha_n}(\eta) \to s_B^{\alpha}(\eta)$, and hence, $s^\alpha_{A_\eta^\alpha}(\eta) \leq s_B^\alpha(\eta)$ for all $B\in \A$.
\end{proof}

\begin{proof}[Proof of Lemma \ref{lem:cond_quantiles_define_valid_quantile_function}]
The claim follows from the particular construction of the rectangular path $(A_\eta^\alpha)$ in Lemma \ref{lem:existence_of_valid_rectangular_array}: We have to show that $F^{-1}$ is increasing and left-continuous. For $0< \alpha < \alpha'< 1$, it holds that $q_\alpha(Y \mid \A)(\omega)\leq q_{\alpha'}(Y \mid \A)(\omega)$ for any $\omega \in \Omega$, since $(A_\eta^\alpha) \subseteq (A_\eta^{\alpha'})$ for any $\eta \in \R$, and hence, $F^{-1}(\alpha)\leq F^{-1}(\alpha')$. Secondly, consider $(\alpha_n)_{n\in \N}\subseteq (0,1)$ such that $\alpha_n \uparrow \alpha \in (0,1)$ as $n\to \infty$. Then, for any $\omega \in \Omega$, by the left-continuity of the rectangular path 
\begin{align*}
    \lim_{n \to \infty}q_{\alpha_n}(Y \mid \A)(\omega)&= \lim_{n \to \infty}\sup \{\eta \in \R\mid \omega \in  A_\eta^{\alpha_n} \}
    =\lim_{n \to \infty} \sup \Big\{\eta \in \R \Mid \omega \in \bigcup_{\substack{\beta \in (0,1)\cap \Q: \beta < \alpha_n}} {A}_\eta^\beta\Big\} \\
    &= \sup \Big\{\eta \in \R \Mid \omega \in \bigcup_{\substack{\beta \in (0,1)\cap \Q: \beta < \alpha}} {A}_\eta^\beta\Big\} 
    = \sup \{\eta \in \R\mid \omega \in  A_\eta^{\alpha} \} = q_{\alpha}(Y \mid \A)(\omega).
\end{align*}
By the equivalent characterization of the class $\GG_\A$ given in Remark \ref{rem:class_GG_A}, it follows directly that the corresponding distribution $F$ is an element in $\GG_\A$.
\end{proof}

\begin{proof}[Proof of Lemma \ref{lem:equivalence_for_conditional_quantiles_given_generated_sigma_lattice}] Assume firstly that $q_\alpha(Y\mid X)=X$. That is $X$ minimizes $\E \QS_\alpha(G,Y)$ over all $\s(X)$-measurable random variables $G$. Since any $\LL(X)$-measurable random variable $G$ is $\s(X)$-measurable (Lemma \ref{lem:sigma_algebras_generated_by_generated_sigma_lattice}) it follows that $X$ minimizes $\E \QS_\alpha(G,Y)$ over all $\LL(X)$-measurable random variables $G$ and hence $q_\alpha(Y\mid \LL(X))=X$. 

For the other direction, recall that, by construction, $q_\alpha(Y\mid \LL(X))=\sup\{\eta \in \R \mid \omega \in A_\eta\}$ for some decreasing minimizing path $(A_\eta)_{\eta \in \R} \subseteq \LL(X)$. We have for all $\eta \in \R$ that
\begin{equation*}
    \{q_\alpha(Y \mid \LL(X)) > \eta\}= \bigcup_{\eta'>\eta} A_{\eta'}\eqand  \{q_\alpha(Y \mid \LL(X)) \geq \eta\}= \bigcap_{\eta'<\eta} A_{\eta'},
\end{equation*}
hence the assumption $q_\alpha(Y \mid \LL(X)) = X$ implies that 
\begin{equation}\label{eq:Ay_rep}
    \{X>\eta\}\subseteq A_\eta \subseteq \{X\geq \eta\}, \eqforall \eta \in \R. 
\end{equation}
Consider now the joint distribution of the random vector $(X,Y)$, $H(x,y)=\prob(X< x, Y < y)$, $x,y \in \R$,
and assume first that $H$ is absolutely continuous. That is, there exists $h:\R^2\to [0,\infty)$ with $H(x,y)=\int_{-\infty}^y \int_{-\infty}^x h(r,s) \diff r \diff s.$ The conditional distribution of $Y$ given $X=x$ equals
\begin{equation}\label{eq_cond_cdf}
    F_{Y \mid X=x}(y) = \frac{1}{h_X(x)} \int_{-\infty}^y h(x,s) \diff s = \frac{\partial_x H(x,y)}{h_X(x)}
\end{equation}
for each $x\in \R$ with positive marginal density $h_X(x)= \int h(x,y) \diff y$. By representation \eqref{eq:generated_sigma_algebra_for_random_variables} of $\LL(X)$, we know that there exists $(a_\eta)_{\eta \in \R} \subseteq \R$ such that $A_\eta= \{X \geq a_\eta\}$ or $A_\eta= \{X > a_\eta\}$ for all $\eta \in \R$. W.l.o.g.\ we let $A_\eta= \{X \geq a_\eta\}$ for all $\eta \in \R$ since $\prob(X=a_\eta)=0$ by assumption. For each $\eta \in \R$, the number $a_\eta$ minimizes
\begin{align*}
    \E\big(\one_{\{X\geq x\}}(\one_{\{Y<\eta\}}-\alpha)\big) &=  \E\big((1-\one_{\{X< x\}})(\one_{\{Y<\eta\}}-\alpha)\big) \\
    &= \prob(Y<\eta)-\alpha-\prob(X<x,Y<\eta)+\alpha \prob(X<x)
\end{align*}
over $x \in \R$ if and only if it minimizes $\alpha H(x,\infty)-H(x,\eta)$ over $x \in \R$, that is
\begin{equation}\label{proof_13}
    \big(\alpha \partial_x H(x,\infty) - \partial_x H(x,\eta)\big)\big|_{x=a_\eta}=0.
\end{equation}
The assumption $q_\alpha(Y\mid \LL(X))=X$ implies $a_\eta=\eta$ for all $\eta \in \R$ by \eqref{eq:Ay_rep}, Thus \eqref{proof_13} implies
\begin{equation}\label{proof_12}
    \alpha h_X(\eta) = \partial_x H(\eta,\eta)=0, \quad \textrm{for all }\eta\in \R, 
\end{equation}
and hence we may conclude by \eqref{eq_cond_cdf} and \eqref{proof_12}
\begin{align*}
    q_\alpha(Y \mid X=x) &= \inf\{y \in \R \mid {\partial_x H(x,y)}/{h_X(x)}\geq \alpha\}
    = \inf\{y \in \R \mid \partial_x H(x,y)\geq h_X(x)\alpha\}=x
\end{align*}
for all $x\in \R$ with $h_X(x)>0$, that is $q_\alpha(Y\mid X)=X$. 

Assume now that the marginal distribution of $X$ has countable support without accumulation points. That is, there exists $\XX=\{x_i \mid i \in I\}$ with a countable (or finite) index sets $I$ such that $\prob(X \in \XX)=1$, $\prob(X = x_i) > 0$ for all $i \in I$, the $x_i$ are pairwise different and they are all isolated points of the support. For any $i\in I$, the conditional distribution of $Y$ given $X=x_i$ equals
\begin{equation*}
    F_{Y \mid X=x_i}(y) = \prob(Y \leq y \mid X=x_i)= \frac{\prob(Y \le y, X = x_i)}{\prob(X = x_i)}.
\end{equation*}
Hence for any $i\in I$, 
\begin{align}\label{proof_cond_quantile}
    q_\alpha(Y \mid X=x_i) &= \inf\left\{y \in \R \mid F_{Y \mid X=x_i}(y)\geq \alpha\right\}\nonumber\\&= \inf\left\{y \in \R \mid  \prob(Y \le y, X = x_i)\geq \alpha \prob(X = x_i)\right\}\nonumber\\&= \inf\left\{y \in \R \mid  \prob(Y < y, X = x_i)\geq \alpha \prob(X = x_i)\right\}.
\end{align}

By representation \eqref{eq:generated_sigma_algebra_for_random_variables} of $\LL(X)$, we know that there exists $(a_\eta)_{\eta \in \R} \subseteq \R$ such that $A_\eta= \{X \geq a_\eta\}$ or $A_\eta= \{X > a_\eta\}$ for all $\eta \in \R$. Since the support of the distribution of $X$ consists of isolated points, we may assume w.l.o.g.\ that $A_\eta= \{X \geq a_\eta\}$ and $a_\eta \in \XX \cup \{-\infty, \infty \}$ for all $\eta \in \R$. The relation \eqref{eq:Ay_rep} implies that $\eta \mapsto a_\eta$ is piecewise constant on $\R \backslash \XX$ and if $(x_i,x_{i'}) \cap \XX = \emptyset$ for $x_i,x_{i'} \in \XX$ with $x_i < x_{i'}$, then $a_\eta = x_{i'}$ for all $\eta \in (x_i,x_{i'})$. For each $\eta \in \R$, the number $a_\eta$ minimizes
\begin{align*}
    \E\big(\one_{\{X\geq x\}}(\one_{\{Y<\eta\}}-\alpha)\big) &=  \E\big((1-\one_{\{X< x\}})(\one_{\{Y<\eta\}}-\alpha)\big) \\
    &= \prob(Y<\eta)-\alpha-\prob(X< x,Y<\eta)+\alpha \prob(X<x) 
\end{align*}
over $x \in \XX \cup \{-\infty, \infty \}$ if and only if it minimizes $\alpha \prob(X< x)-\prob(X<x,Y<\eta)$. That is, for each $\eta \in \R$, $a_\eta \in \XX \cup \{-\infty, \infty\}$ satisfies 
\begin{equation}\label{eq:10}
       \alpha \prob(X <  a_\eta)-\prob(X <  a_\eta,Y < \eta) \leq \alpha \prob(X < x_k)-\prob(X < x_k,Y < \eta) \eqforall k \in I.
\end{equation}

Let $x_i,x_{i'},x_{i''} \in \XX \cup \{-\infty,\infty\}$ with $x_{i''} < x_i < x_{i'}$ and $(x_{i''},x_i) \cap \XX = \emptyset$, $(x_i,x_{i'}) \cap \XX = \emptyset$. Condition \eqref{eq:10} for $\eta \in (x_i,x_{i'})$ and $k = i$ implies $\alpha \prob(X = x_i) \le \prob(X = x_i,Y < \eta)$, and for $\eta \in (x_{i''},x_i)$ and $x_k = x_{i'}$, it yields $\alpha \prob(X = x_i) \ge \prob(X = x_i, Y < \eta)$, hence, by \eqref{proof_cond_quantile}, we obtain $q_\alpha(Y \mid X=x_i) = x_i$.

\end{proof}

\end{document}